\documentclass[english]{amsart}
\usepackage[latin9]{inputenc}
\usepackage{prettyref}
\usepackage{float}
\usepackage{amsthm}
\usepackage{amssymb}
\usepackage{graphicx}

\makeatletter
\numberwithin{equation}{section}

\usepackage{amsthm}
\usepackage{amssymb}
\usepackage{amsopn}
\usepackage{amscd}
\usepackage{amsfonts}
\usepackage{graphicx}
\usepackage[percent]{overpic}

\newcommand{\R}{\mathbb R}
\newcommand{\C}{\mathbb C}
\newcommand{\CHat}{{\widehat \C}}
\newcommand{\HP}{\mathbb H}

\newcommand{\N}{\mathbb N}
\newcommand{\Q}{\mathbb Q}
\newcommand{\Z}{\mathbb Z}

\theoremstyle{theorem}
\newtheorem{theo}{Theorem}
\newtheorem{lema}{Lemma}
\newtheorem{prop}{Proposition}
\newtheorem{coro}{Corollary}

\theoremstyle{definition}
\newtheorem{defi}{Definition}
\newtheorem{exa}{Example}

\theoremstyle{remark}
\newtheorem*{rema}{Remark}
\newtheorem*{note}{Note}

\makeatother

\usepackage{babel}

\usepackage{chngcntr}
\counterwithout{figure}{section}

\begin{document}

\title{Dynamical aspects of piecewise conformal maps}

\author{Renato Leriche$^1$}
\address{{[}Renato Leriche$^1${]} }
\email{r\_lerichev@ciencias.unam.mx}

\author{Guillermo Sienra$^1$}
\address{{[}Guillermo Sienra$^1${]} }
\email{gsl@dinamica1.fciencias.unam.mx}

\address{$^1$ Depto. de Matem\'aticas, Fac. Ciencias, UNAM (Universidad Nacional Aut\'onoma de M\'exico), Ciudad de M\'exico, M\'exico.}



\begin{abstract}
We study the dynamics of piecewise conformal maps in the Riemann sphere. The normality and chaotic regions are defined and we state several results and properties of these sets. We show that the stability of these piecewise maps is related to the Kleinian group generated by their transformations  under certain hypotheses.
The general motivation of the article is to compare the dynamics of piecewise conformal maps and those of the Kleinian groups and iterations of rational maps.
\end{abstract}

\thanks{\emph{Acknowledgments.} This work was partially supported by PAPIIT IN 102515 and CONACYT CB15/255633.}

\maketitle

\section{introduction}

The study of the dynamics of piecewise maps comes from a variety of
contexts, such as the piecewise interval maps (see for instance \cite{V},
\cite{C}), the piecewise maps of isometries of the plane (see for
instance \cite{G1}, \cite{G2}) and some applications of differentiable
piecewise maps (see for instance \cite{BBCK}, \cite{PL}).

However, our main motivation, is to extend the so called Sullivan
dictionary of conformal dynamics to piecewise conformal maps. In such
dictionary there are involved the dynamics of iterations of holomorphic
maps of the Riemann sphere and the dynamics of Kleinian groups. In
both dynamics there is a duality in the behavior of the orbits of
points: the conservative and the dissipative part, being both, invariant
sets under the dynamics. The conservative part is called the Julia
set in holomorphic dynamics and it is called the limit set in the
dynamics of Kleinian groups and it is where the dynamics is more interesting.
The Fatou set, in the other hand, is the region of normality for
the set of iterations of a map, see \cite{M}.


In this paper we are interested in piecewise maps on the Riemann sphere
$\CHat$ which are restrictions of conformal automorphisms in each
piece. The equivalent to the Julia set, in this context, is the pre-discontinuity
set, following \cite{Cr}, and its complement is the regularity region.
In Section \prettyref{subsec:Generalities} we include the relevant
definitions. The regularity regions are very well known in the case
of iterations of holomorphic maps (\cite{Be}), in our case is less
complicated and \prettyref{thm:Classif-Fatou} in Section \prettyref{subsec:Classification-of-Fatou}
gives the classifications of the regularity regions. A result in holomorphic
dynamics shows that the connectivity of a regularity region (Fatou
domains) is 0, 1 or ${\infty}$, in Section \prettyref{subsec:Connectivity}
we show that in the iteration of piecewise conformal maps, the regularity
region can have any connectivity, from zero to infinity. In Section
\prettyref{subsec:Symbolic-dynamics}, we explain the symbolic dynamics
related to a piecewise map.

A celebrated theorem of Sullivan (\cite{S}), shows that in the dynamics
of rational maps, there are not wandering Fatou domains, although
for certain entire transcendental maps there are (\cite{Be}). In
Section \prettyref{subsec:Wandering-Domains} we show that it is the
case that there are piecewise conformal maps with wandering domains. In holomorphic dynamics we can find Julia sets which are the whole Riemann sphere. We show in Section \prettyref{subsec:Whole-sphere}
an example of piecewise conformal map whose pre-discontinuity set is the whole
sphere.

In \prettyref{sec:Relation-to-Kleinian}, we relate the dynamics of
a piecewise conformal map $F$ to a certain Kleinian group. To do so, we consider
the extension of each conformal automorphism to the whole sphere and
we consider the group $\Gamma_{F}$ generated by these extended maps.
\prettyref{thm:Alpha-Limit} and \prettyref{thm:Omega-Limit} shows
the relation between the $\alpha$ and $\omega$-limits of $F$ and
the limit set of the group $\Gamma_{F}$, ${\Lambda}(\Gamma_{F})$.

In \prettyref{sec:Deformations-and-Stability}, we study some aspects
of the deformations and stability of a given piecewise conformal map. \prettyref{thm:ContDefSpidN}
and \prettyref{thm:ContDefSpid} show that if we deform in a continuous
manner the boundary of the regions initially involved, then the pre-discontinuity set deforms
also in a continuous way, subject to certain hypothesis. Finally,
\prettyref{thm:StrucStabSchottky} gives conditions, related to the
group $\Gamma_{F}$, for a piecewise conformal map to be structurally
stable.

We conclude this paper with two complementary sections. \prettyref{sec:Examples}
deals with some examples and images of the theory.
Finally \prettyref{sec:Technical-lemmas} contains some technical
results as well as the combinatorics on the pre-discontinuity set.

\section{Piecewise Conformal Maps\label{sec:Piecewise-Conformal-Maps}}

\subsection{Generalities\label{subsec:Generalities}}

In this paper we will denote by ${\C}$ the complex plane and $\CHat$
the Riemann sphere. We begin this section with the definition of a
piecewise map in the sphere and we define the regions of conformality
and discontinuity.

\begin{defi}A \textbf{piecewise conformal map}, is a pair $(P,F)$,
where:
\begin{enumerate}
\item $P=\left\{ R_{m}\right\} $ is a finite partition of $\CHat$,
that is, $\bigcup_{m}R_{m}=\CHat$ and $R_{m}\cap R_{n}=\emptyset$
for all $m\neq n$. Additionally:
\begin{enumerate}
\item The interior of each $R_{m}$ is non empty.
\item The boundary of each $R_{m}$, denoted $\partial R_{m}$, is a finite
union of simple closed curves.
\end{enumerate}
\item $F:\CHat\rightarrow\CHat$ is a map satisfying that $F:R_{m}\rightarrow F(R_{m})$
is the restriction of a conformal automorphism of $\CHat$.
\end{enumerate}
\end{defi}

Following \cite{Cr}, associated to any piecewise conformal map, there
is a \textbf{set of discontinuity} denoted by $\partial R$, the set
$\bigcup_{m}\partial R_{m}$, which is the union of the boundaries
of the sets $R_{m}$. Whereas that $\bigcup_{m}\mathrm{int}(R_{m})$,
the union of interior of the sets $R_{m}$, is the \textbf{region
of conformality}.

For a set $A\subset\CHat$, define $F^{-n}(A)=\left\{ z\in\CHat:F^{n}(z)\in A\right\} $.
We have the next

\begin{defi}The \textbf{pre-discontinuity set}\footnote{This set was originally
named "Spider Web" (see \cite{Cr}) because of the visual resemblance in some cases, but being mathematically unintuitive we prefer the one given in this paper.} of a piecewise conformal map
$F$ is the set 
\[
\mathcal{PD}(F)=\overline{\bigcup_{n\geq0}F^{-n}(\partial R})
\]

\end{defi}

The pre-discontinuity set of $F$ has a natural stratification by the subsets
$\partial R=\mathcal{PD}_{0}(F)\subset \mathcal{PD}_{1}(F)\subset...\subset \mathcal{PD}_{n}(F)\subset...\subset \mathcal{PD}(F)$,
where $\mathcal{PD}_{n}(F)=\bigcup_{k=0}^{n}F^{-k}(\partial R)$, for $n\geq0$.

\begin{defi}The $\alpha$\textbf{-limit set }of a piecewise conformal
map $F$ is 
\[
\alpha(F)=\mathcal{PD}(F)-\bigcup_{n\geq0}F^{-n}(\partial R)
\]

\end{defi}

Hence we have that for any $z\in \mathcal{PD}(F)-\alpha(F)$, $z\in \mathcal{PD}_{n}(F)$
for some natural $n\geq0$ and so $F^{n}(z)\in\partial R$, the discontinuity
set.

From the discussion above, it follows that the family $\{F^{n}(z)\}_{n=1}^{\infty}$
is not normal if and only if $z\in \mathcal{PD}(F)$. Hence the pre-discontinuity set
is the equivalent to the Julia set for holomorphic dynamics. The complement
of the pre-discontinuity set is consequently the \textbf{region of regularity},
equivalently, the \textbf{Fatou set} as it is called in holomorphic
dynamics. A good reference for definitions and properties of those
sets in rational dynamics is the book of Milnor (\cite{M}). For further
properties about the pre-discontinuity set and the Fatou set of piecewise conformal
maps, see \prettyref{sec:Technical-lemmas}.



\subsection{Classification of Fatou components.\label{subsec:Classification-of-Fatou}}

In this section we classify the periodic regions of regularity when
each of the maps involved are conformal automorphisms.

If $U$ is a connected component of the Fatou set, we say that $U$
is periodic if $F^{n}(U)\subset U$ for some $n>0$. The minimum number
$n$ with that property is the period of the component. Then for $U$
a connected component of the Fatou set, either there exists $k>0$
such that $F^{k}(U)$ is periodic or does not exist such $k$. In
the first case we say that $U$ is preperiodic and in the second $U$
is wandering.

The classification of periodic Fatou components of piecewise conformal
automorphisms turns out to be rather simple and it is as follows (see
\prettyref{sec:Examples} for examples).

Observe that if $U$ is periodic and $F^{n}(U)\subset U$, then $F^{n}|_{U}:U\rightarrow U$
is a composition of conformal maps hence it is the restriction of
a M{\"o}bius transformation.

\begin{theo}\label{thm:Classif-Fatou}Let $F$ be a piecewise conformal
map and $U$ a periodic component of $F$ of period $n$, then we
have the following cases:

\emph{(i)} There is a fixed point under $F^{n}|_{U}$ inside $U$.
In this case either $F^{n}|_{U}$ is \emph{(a)}: a loxodromic transformation
with its attracting fixed point in $U$, \emph{(b)}: $F^{n}|_{U}$
is an elliptic transformation with at least one of their fixed point
inside $U$, or \emph{(c)}: $F^{n}|_{U}$ is the identity.

\emph{(ii)} There is a fixed point of $F^{n}|_{U}$ in the boundary
of $U$, so \emph{(a)}: $F^{n}|_{U}$ is an hyperbolic transformation,
\emph{(b)}: $F^{n}|_{U}$ is a parabolic transformation, or \emph{(c)}:
$F^{n}|_{U}$ is the identity.

\emph{(iii)} There is not a fixed point of $F^{n}|_{U}$ in $U$,
so $F^{n}|_{U}$ is an elliptic transformation.

\end{theo}

\begin{proof} The map $F^{n}|_{U}$ $:U\rightarrow U$ is a M\"obius
transformation. Recall that M\"obius transformations are classified
in loxodromic, parabolic and elliptic, see \cite{A} for instance.
First, let us assume that $F^{n}|_{U}$ is not the identity map.

Forward invariant open sets (that is $F^{n}(U)\subset U$) of loxodromic
transformations must contain the attracting fixed point or have it
on its boundary, that is case (ia) or (iia).

Since the boundary of an open Fatou connected domain is contained
in the pre-discontinuity set of $F$, then for a parabolic transformation, its
fixed point is in the boundary of the domain, that is (iib).

If the transformation is elliptic, then the Fatou domain must be invariant
under rotations and that is (ib) if it contains one fixed point or
(iii) if not. If $F^{n}|_{U}$ is the identity map then we have (ic)
or (iic). \end{proof}

\begin{rema}In case (ia) for all $z\in U$, then $F^{nk}(z)$ tends
to $p$, the fixed point of $F^{n}|_{U}$, when $k$ tends to $\infty$
(see figure \ref{fig:attr}).

In case (ib) $F^{n}|_{U}$ is periodic if the angle of rotation of
$F^{n}|_{U}$ is rational (see figure \ref{fig:rot}) or it is quasiperiodic
if the angle of rotation of $F^{n}|_{U}$ is irrational (see figure
\ref{fig:rotirr}).

In case (iia) or (iib) for all $z\in U$, we have that $F^{nk}(z)$
tends to $p$, the fixed point of $F^{n}|_{U}$, when $k$ tends to
$\infty$ (see figure \ref{fig:parab}).

Case (iii) behaves like (ib) but without fixed points inside $U$.\end{rema}

\begin{rema}About elliptic transformations, we can do a more detailed
analysis. A $F^{n}|_{U}$ elliptic can contain two fixed points, but $U\neq \CHat$ then is not simply connected (see figure \ref{fig:rot2fix}).

If $F^{n}|_{U}$ is elliptic without fixed points in $U$, then such component
is not simply connected because $U$ must contain one invariant simple closed curve $C$
separating the fixed points (see figure \ref{fig:rotann}).

If $U$ is a simply connected periodic component and $F^{n}|_{U}$
is not loxodromic, parabolic or elliptic with one fixed point in $\overline{U}$,
then $F^{n}|_{U}$ can not be elliptic with fixed points outside $U$.
Therefore for such $U$, $F^{n}|_{U}=Id$
(see figures \ref{fig:rotneutr} and \ref{fig:rotext}).\end{rema}


\subsection{Connectivity of the Fatou components.\label{subsec:Connectivity}}

In rational dynamics it is known that the connectivity of the Fatou
set is one, two or infinity, see \cite{M}. Here we show that for
piecewise conformal dynamics the connectivity of the regularity set
can be any natural number or infinity.

\begin{exa}For $k$ a positive natural number, let $D_{k}$ the disc
with center at $z=k$ and small radius, say $r=(1/2)(k\tan(\pi/k))$,
$S_{k}$ its boundary and $D_{k}^{c}$ its complement. Define $g(z)=2z,$
if $z\in D_{k}$ and $f(z)=e^{2\pi i/k}z$ if $z\in D_{k}^{c}$. The
map $f(z)$ is a rotation. The map $F$ is generated by $f$ and $g$
.

Observe that the set $\left\{ f^{-j}(S_{k})\right\} _{j=1}^{k}$of
$k$ circles is contained in $\mathcal{PD}(F)$ and in fact $\mathcal{PD}(F)\subset\cup_{j=1}^{k}f^{-j}(D_{k})$
. That means that the complement of the $k$ discs $\left\{ f^{-i}(D_{k})\right\} _{j=1}^{k}$
is a region of regularity of $F$ with $0$ and $\infty$ as elliptic
fixed points. Such region has connectivity $k.$ The result does not
depend on the choice for $g$. See figure \ref{fig:conn}.\end{exa}

It is left to show that there is a region of regularity with infinity
connectivity.

\begin{exa}Consider $D$ the disc with center at $1$ and radius
$1/3.$ Choose $g(z)$ any M\"obius transformation, if $z\in D$
and $f(z)=2z$ if $z$ is in the complement of $D$. The map $F$
is generated by $f$ and $g$ .

Notice that the set $\left\{ f^{-j}(D)\right\} _{j=1}^{\infty}$ is
a disjoint set of discs converging to $0$ and of radius tending to
$0$. It is clear that $\mathcal{PD}(F)\subset\cup_{j=1}^{\infty}\left\{ f^{-j}(D)\right\} $
. The complement of such sequence of disc is a region of regularity
of $F$ with infinite connectivity. See figure \ref{fig:conninfty}.\end{exa}

\subsection{Symbolic dynamics.\label{subsec:Symbolic-dynamics}}

For a piecewise conformal automorphism $(P,F)$, the partition $P=\left\{ R_{m}\right\} _{m=1}^{M}$
leads naturally to a coding map, following the reasoning in \cite{G1}
for piecewise isometries. The coding space is the set of infinite
sequences of $M$ symbols $\Sigma_{M}=\left\{ 1,\dots,M\right\} ^{\N}$
and the coding map $I_{F}:\CHat\rightarrow\Sigma_{M}$ is defined
as $I_{F}(z)_{k}=m$ if $F^{k}(z)\in R_{m}$, where $I_{F}(z)_{k}=m_{k}$
is the $k$-th entry of the coding generated sequence $I_{F}(z)=(m_{1},m_{2},m_{3},\dots,m_{k},\dots)$.
The map $I_{F}$ is called \textbf{itinerary} because encodes the
forward orbit of a point by recording the indexes of visited sets
$R_{m}$.

A coding partition in cells $\left\{ \mathcal{C}_{s}\right\} _{s\in\Sigma_{M}}$
over $\CHat$ is induced by the equivalence relation $z\thicksim w$
if and only if $I_{F}(z)=I_{F}(w)$. As we can see from \cite{Cr},
the pre-discontinuity set is $\overline{\bigcup_{s\in\Sigma_{M}}\partial\mathcal{C}_{s}}$
and the Fatou set is $\bigcup_{s\in\Sigma_{M}}\mathrm{int}(\mathcal{C}_{s})$,
the union of interior of the partition cells. This give us a relevant
fact: components of Fatou set contains only points with the same itinerary
(see figure \ref{fig:itin}).

The shift map $\sigma:\Sigma_{M}\rightarrow\Sigma_{M}$, $\sigma(m_{1},m_{2},m_{3}\dots)=(m_{2},m_{3}\dots)$
is semi-conjugated to $F$ because $I_{F}\circ F(z)=\sigma\circ I_{F}(z)$
but $I_{F}$ is many-to-one for most of the sphere, since each partition
cell could contain more than one point.

We classify the points of $\CHat$ in two sets according to the
itinerary. Points with rational itinerary are those belonging to periodic
or eventually periodic cells. Points with irrational itinerary, set
denoted as $\mathcal{I}(F)$, are in wandering cells. By definition,
both of rational and irrational sets are invariants.

\bigskip

\begin{rema}If $F$ has a wandering domain $U$, then $U\subset\mathcal{I}(F)$,
because the itinerary of each point in $U$ is not preperiodic.\end{rema}

\begin{rema}According to \cite{G1}, the exceptional set is defined
as $E=\mathcal{I}(F)$ for piecewise isometries and prove that on
a space of finite Lebesque measure, every cell $\mathcal{C}$ of positive
measure is eventually periodic. As for each of such cells $\mathrm{int}(\mathcal{C})$
is a subset of the Fatou set, then $E\subset \mathcal{PD}(F)$ in case of
$F$ piecewise isometry defined in a subset of finite Lebesgue measure
of $\C$.\end{rema}

\subsection{Wandering Domains.\label{subsec:Wandering-Domains}}

In this section we will show that there are piecewise conformal maps
with wandering domains. We present two examples.

\begin{exa}Here, our construction relies in \cite{C} Theorem A,
that proves the existence of a wandering domain for an affine interval
exchange transformation of an interval. We will explain the basic
facts and extend the construction to piecewise conformal maps.

Consider $I=[0,1)\subset{\R}$ and $0=y_{0}<y_{1}<y_{2}<...<y_{m-1}<y_{m}=1$
a finite sequence. We say that $T:[0,1)\rightarrow[0,1)$ is an affine
interval exchange transformation (AIET) if for all $i\in\{0,1,...,m-1\}$,
the restriction $T|_{[y_{i},y_{i+1})}$ is continuously differentiable
and its derivative is identically equal to a constant $\beta_{i}>0$.
Hence $T|_{[y_{i},y_{i+1})}(x)=\beta_{i}x+r_{i,}$, with $r_{i}$
a real number. The AIET in Theorem A of \cite{C} has $m=4$, the
$y_{i}$ depending on the derivatives $\beta_{j}$ which in turn depend
on a certain Perron-Frobenius matrix.

To construct our example consider the following piecewise conformal
dynamical system $F$:

Fix $y_{i}$, $\beta_{i}$ and $r_{i}$, $i\in\{0,1,2,3\}$ as in
Theorem A of \cite{C}. Let the set of discontinuity be the union
of the lines $L_{i}=\left\{ z:Re(z)=y_{i},i\in\{0,1,2,3\}\right\} $.
Let $F$ be such that $F(z)=z$ if $Re(z)<0$ or $Re(z)>1$ and $F(z)=\beta_{i}z+r_{i}$
if $y_{i}<Re(z)<y_{i+1}.$ Observe that since each line $L_{i}$ is
orthogonal to the real line then each line in $\mathcal{PD}(F)$ is orthogonal
to the real line and its complement is a union of vertical strips.
The restriction of such complement to the real line contains the wandering
interval of $T$, therefore $F$ has a wandering strip.

We can construct a similar piecewise conformal transformation using
discs instead of strips as follow: For consecutive $y_{i}$ and $y_{i+1}$
consider the disc $R_{i}$ with diameter $y_{i+1}-y_{i}$ and center
$\frac{1}{2}(y_{i}+y_{i+1})$. Let $F$ such that $F(z)=\beta_{i}z+r_{i}$
in each disc $R_{i}$ and $F(z)=z$ outside all of discs. Then the
associated pre-discontinuity set is an infinite union of arcs of circumferences.
As in the previous case, the restriction to the real line must contain
the wandering domain inherited from the AIET and then $F$ has a wandering
component of the Fatou set.

In the terminology of \cite{AG}, our map $F$ is close to the concept
of a cone exchange transformation, except that we are allowing affine
interval exchanges rather than the more restrictive isometric interval
exchanges.

Other version to \cite{C} with different properties can be found
in \cite{B} for instance.\end{exa}

\begin{exa}Here, we show that there exist a piecewise conformal map with all of the
components of the regular set wandering. Let be $R=\left\{ z:\,Im(z)<0\right\} $,
$T|_{R}(z)=f(z)=iz$ and $T|_{R^{c}}(z)=g(z)=-iz+1+i$. Notice that
$f$ and $g$ are both euclidean rotations. First, we have $\mathcal{PD}(T)=\left\{ z:\,Re(z)\in\Z\,\mathrm{or}\,Im(z)\in\Z\right\} $,
then the Fatou set is formed by open squares which elements have not
integer coordinates.

Let be $c_{n}=a+bi\in(0,1)\times(n,n+1)\subset R^{c}$ , where $n\in\N$.
Calculating the orbit of $c_{n}$: 
\[
\begin{aligned}T(c_{n}) & = & g(c_{n}) & = & b+1+(1-a)i & \in R^{c}\\
T^{2}(c_{n}) & = & g\circ g(c_{n}) & = & 2-a-bi & \in R\\
T^{3}(c_{n}) & = & f\circ g\circ g(c_{n}) & = & b+(2-a)i & \in R^{c}\\
 & \vdots\\
T^{2n+3}(c_{n}) & = & (f\circ g)^{n+1}\circ g(c_{n}) & = & b-n+(n+2-a)i & \in(0,1)\times(n+1,n+2)
\end{aligned}
\]

Then, the itinerary of $c_{n}$ is 
\[
(1,\overbrace{1,0}^{n+1\,\mathrm{times}},1,\overbrace{1,0}^{n+2\,\mathrm{times}},1,\overbrace{1,0}^{n+3\,\mathrm{times}},\dots),
\]
clearly a irrational sequence and in consequence, as we saw in Section
\ref{subsec:Symbolic-dynamics}, the square-component containing $c_{n}$
is wandering (see figure \ref{fig:wander}).

Let be $Q_{I}=(0,\infty)\times(0,\infty)\subset R^{c}$. The transformation
$f\circ g$ is the translation $z\mapsto z-1+i$, which is applied
to points $z\in Q_{I}$ with $Re(z)>1$ and $Im(z)>0$, whose orbit
must reach a wandering square $(0,1)\times(n,n+1)\subset Q_{I}$.
Then, all components in $Q_{I}$ are wandering. Define now $Q_{II}=(0,\infty)\times(-\infty,0)\subset R$,
$Q_{III}=(-\infty,0)\times(-\infty,0)\subset R$ and $Q_{IV}=(-\infty,0)\times(0,\infty)\subset R^{c}$.
Observe that $T(Q_{III})=Q_{II}$, $T(Q_{II})=Q_{I}$, $T(Q_{IV})\subset Q_{I}$
and $T(Q_{I})\subset Q_{I}\cup Q_{II}$, then we can conclude that
the orbits for all points visit the set $Q_{I}$, and in consequence,
all components in the regular set must be wandering.\end{exa}

A well known theorem of Sullivan establish that components in the
Fatou set of rational functions in the Riemann sphere are not wandering
(see \cite{S}). About piecewise maps exists some results in this
direction. When $X$ is a metric space with finite Lebesgue measure
and $F:X\rightarrow X$ is a piecewise isometry, then every component
in the Fatou set is eventually periodic (see \cite{G1}, Proposition 6.1). From the
later theorem we have the following

\begin{coro}If $F:\CHat\rightarrow\CHat$ is a piecewise conformal
map where $\partial R$ is bounded, $\infty\in R_{1}$, $F|_{R_{1}}$
is an euclidean rotation and $F|_{R_{m}}$ is an euclidean isometry
for $m>1$, then every component in the Fatou set is eventually periodic.\end{coro}

\begin{proof}Since $\partial R$ is a bounded set, exists a disc
$D\subset\C$ centered in the finite fixed point of $F|_{R_{1}}$
such that $\partial R\subset\overline{R_{1}^{c}}\subset D$ and $F^{n}(R_{m}\cap D)\subset D$
for each $R_{m}$ and for all $n$. Then $F|_{D}$ is a piecewise
(euclidean) isometry in a finite Lebesgue measure set and the result
follows.\end{proof}

In relation with the pre-discontinuity set, we establish the next

\begin{prop}If $F$ is a piecewise conformal map such that $\mathcal{PD}(F)=\mathcal{PD}_{N}(F)$
for some $N$, then every component in the Fatou set is eventually
periodic.\end{prop}

\begin{proof}$\mathcal{PD}(F)=\mathcal{PD}_{N}(F)$ is union of
boundaries of a finite number equivalence classes from itineraries, in consequence the Fatou set is composited by interiors a finite number of equivalence classes.
Therefore, for each component $U$, its orbit
$\left\{ F^{n}(U)\right\} $ in contained in a finite number of components.\end{proof}


\subsection{ A Piecewise conformal map for which the pre-discontinuity set is the entire sphere.\label{subsec:Whole-sphere}}


\begin{exa}Let be $R$ the disc with center in $0$ and radius $1$
and $F$ the piecewise conformal map defined with $F|_{R}(z)=2z$
and $F|_{R^{c}}(z)=\frac{2}{3}z$. We claim that $\mathcal{PD}(F)=\CHat$.

First, we can notice that $F$ can not have periodic points $z\neq0$.
Otherwise, if $F^{n}(z)=z\neq0$ with $n\geq1$, then $F^{n}(z)=2^{i}(\frac{2}{3})^{j}z=z$
and $2^{i+j}=3^{j}$, clearly a contradiction.

Second, we will show that $\big(\bigcup_{n\in\mathbb{N}}F^{-n}(\partial R)\big)\cap[0,\infty)=\bigcup_{n\in\mathbb{N}}A_{n}$,
where \linebreak{}
$A_{n}=\left\{ \frac{1}{2^{n}},\frac{3}{2^{n}},\dots,\frac{3^{n}}{2^{n}}\right\} $.
Let $q=\frac{3^{m}}{2^{n}}$ with $0\leq m\leq n$ (that is, $q\in A_{n}$).
We easily check the following statements:
\begin{enumerate}
\item If $q=1$, then $q\in\partial R$. Also note that $A_{0}=\left\{ 1\right\} $.
\item If $q>1$, then $F(q)=\frac{2}{3}q=\frac{3^{m-1}}{2^{n-1}}$ and clearly
$F(q)\in A_{n-1}\cup\left\{ 1\right\} $.
\item If $q<1$ then $F(q)=2q=\frac{3^{m}}{2^{n-1}}$.
\begin{enumerate}
\item If $m\leq n-1$ then $F(q)\in A_{n-1}\cup\left\{ 1\right\} $.
\item If $m>n-1$, then $n=m$ , but this is impossible because in those
case $\frac{3^{m}}{2^{n}}>1$ and by hypothesis $q<1$.
\end{enumerate}
\end{enumerate}
Since $q$ can not be periodic, exists $N\geq0$ such that $F^{N}(q)=1\in\partial R$,
that is, $q\in \mathcal{PD}_{N}(F)$.

Third, $F|_{[\frac{2}{3},2)}$ is an affine interval exchange transformation:
$F([\frac{2}{3},1))=[\frac{4}{3},2)$ and $F([1,2))=[\frac{2}{3},\frac{4}{3})$.
An affine interval exchange transformation $f:[0,1)\rightarrow[0,1)$
with $f|_{[0,c)}(x)=\lambda x+a$ and $f|_{[c,1)}(x)=\mu x+b$ is
conjugated to the rotation $\tau_{\theta}:S^{1}\rightarrow S^{1}$
of angle $\theta=\frac{\log\lambda}{\log\lambda-\log\mu}$ (see \cite{Liou}).
In our case, $F|_{[\frac{2}{3},2)}$ is conjugated to $f:[0,1)\rightarrow[0,1)$
with $f|_{[0,\frac{1}{4})}(x)=2x+\frac{1}{2}$ and $f|_{(\frac{1}{4},1)}(x)=\frac{2}{3}x-\frac{1}{6}$,
and then conjugated to the rotation of angle $\theta=\frac{\log2}{\log2-\log(2/3)}=\frac{\log2}{\log3}\notin\Q$,
in consequence every orbit of $x\in[\frac{2}{3},2)$ is dense in such
interval.

Let be $x\in(0,\infty)$. Then exists $N\geq0$ such that $F^{N}(x)\in[\frac{2}{3},2)$,
because $F$ is expansive if $x<\frac{2}{3}$ and contractive if $x>2$.
Since the orbit of $F^{N}(x)$ is dense in $[\frac{2}{3},2]$, exists
$M$ such that $F^{N+M}(x)\in(1-\delta_{1},1+\delta_{2}$), where
$\delta_{1}=\frac{\varepsilon}{x+\varepsilon}$ and $\delta_{2}=\frac{\varepsilon}{x-\varepsilon}$
for a given $\varepsilon>0$. 

$F^{N+M}(x)=2^{i}(\frac{2}{3})^{j}x=\frac{2^{i+j}}{3^{j}}x$, where
$i+j=N+M$. Then
\[
\begin{array}{rccclc}
1-\frac{\varepsilon}{x+\varepsilon} & < & \frac{2^{i+j}}{3^{j}}x & < & 1+\frac{\varepsilon}{x-\varepsilon} & \implies\\
\frac{x}{x+\varepsilon} & < & \frac{2^{i+j}}{3^{j}}x & < & \frac{x}{x-\varepsilon} & \implies\\
\frac{1}{x+\varepsilon} & < & \frac{2^{i+j}}{3^{j}} & < & \frac{1}{x-\varepsilon} & \implies\\
x+\varepsilon & > & \frac{3^{j}}{2^{i+j}} & > & x-\varepsilon
\end{array}
\]

That is, for a given $\varepsilon>0$ exists $y\in\bigcup_{n\in\mathbb{N}}A_{n}$
such that $y\in(x-\varepsilon,x+\varepsilon)$. Then, $\mathcal{PD}(F)\cap[0,\infty)=[0,\infty)$.
Since $F$ behaves the same in each ray from origin, we have $\mathcal{PD}(F)=\CHat$.\end{exa}

\subsection{Piecewise Hyperbolic 3D Isometries.\label{subsec:Piecewise-Hyperbolic-3D}}

We want to point out that Poincar\'e showed that any conformal automorphism
of the sphere (a M\"obius transformation), extends to the interior
of the $3$-hyperbolic space, $\HP^{3}$, as an isometry. Hence, given
any piecewise conformal automorphism $F$ on the sphere with regions
$R_{i}$, it extends to a piecewise isometry of the hyperbolic space.
To do so, extend each region $R_{i}$ to a regions $S_{i}$ in $\HP^{3}$,
such that, $S_{i}\cap{\partial}\HP^{3}=R_{i}$ and extend each of
the the corresponding transformations.

Interesting cases comes when the regions $R_{i}$ are restrictions
of discs, since we can consider their natural extensions, the spherical
gaskets. Such spherical gaskets are sent to spherical gaskets
under any M{\"o}bius transformation, and the pre-discontinuity set of the extended
map $F$ will be made of pieces of spherical gaskets over the original
pre-discontinuity set. However, we do not pursue this topic in this article.

\section{Relation to Kleinian Groups\label{sec:Relation-to-Kleinian}}

Given a piecewise conformal map $F$, we can naturally associate the
finitely generated subgroup $\Gamma_{F}=\left\langle F|_{R_{m}}\right\rangle $
of $PSL(2,\C)$, since each $F|_{R_{m}}$ is a M\"obius transformation.
If $\Gamma_{F}$ is discrete, then is called Kleinian group. Such
group also act as a discrete group of isometries in $\mathbb{B}^{3}$, 
the hyperbolic open unit $3$-ball. Recall that the limit set $\Lambda(\Gamma)$
of a Kleinian group $\Gamma$ is the set of accumulation points of
$\Gamma p$, where $p\in\mathbb{B}^{3}$. $\Lambda(\Gamma)\subset\CHat$
because $\Gamma p$ accumulates in $\partial\mathbb{B}^{3}=\mathbb{S}^{2}\cong\CHat$.
The regular set of $\Gamma$ is $\Omega(\Gamma)=\CHat-\Lambda(\Gamma)$.
The limit set $\Lambda(\Gamma)$ can also be characterized as the
limit of sequences of distinct elements $\gamma_{i}\in\Gamma$ applied
to a any element $z\in\Omega(\Gamma)$.

In relation to the associated Kleinian group, we can demonstrate that
if $\partial R$ is contained in the region of regularity of $\Gamma_{F}$,
then the $\alpha$-limit of $F$ lands in the limit set of $\Gamma_{F}$.

\begin{theo}\label{thm:Alpha-Limit}If $\partial R\cap\Lambda(\Gamma_{F})=\emptyset$,
then
\begin{enumerate}
\item $\alpha(F)\subset\Lambda(\Gamma_{F})$, and
\item $\alpha(F)=\underset{n\rightarrow\infty}{\lim}\overline{F^{-n}(\partial R)}$
in $\mathcal{H}(\CHat)$.
\end{enumerate}
\end{theo}

\begin{note} $\mathcal{H}(\CHat)$ is the space of compact subsets
of $\CHat$ with the Hausdorff topology induced from a spherical
metric in $\CHat$, see \prettyref{sec:Technical-lemmas}.\end{note}

\begin{proof}Let $L=\underset{n\rightarrow\infty}{\lim}\overline{F^{-n}(\partial R)}$.
If $L=\emptyset$, then $\alpha(F)=\emptyset$, by Proposition \prettyref{prop:AlphaLimit1}.

Suppose $L\neq\emptyset$ and let $z\in L$, $f_{m}=F|_{R_{m}}$ and
$\Sigma_{M}(k)$ the set of words of $k$ length of $M$ symbols.
Note that 
\[
\overline{F^{-1}(\partial R)}=\bigcup_{m=1}^{M}\big(f_{m}^{-1}(\partial R)\big)\cap\overline{R_{m}}=\overline{C_{1,1}}\cup\dots\cup\overline{C_{1,M}}=\bigcup_{t\in\Sigma_{M}(1)}\overline{C_{1,t}}
\]
and consequently $\overline{F^{-n}(\partial R)}=\bigcup_{t\in\Sigma_{M}(n)}\overline{C_{n,t}}$,
where each $C_{n,t}$ is a finite union of curve segments and points,
or an empty set (see last remark in \prettyref{sec:Technical-lemmas}).

At least one $C_{n,t}$ is not empty for each $n$ level, because
we assumed $L\neq\emptyset$. Then we can take $z_{n}\in C_{n,t}\subset\overline{F^{-n}(\partial R)}$
such that $z_{n}\rightarrow z$, because every neighborhood of $z$
intersects infinitely many $\overline{F^{-n}(\partial R)}$. Even
more, we can choose $z_{n}\in C_{n,t}$ such that $F(z_{n})=z_{n-1}$,
because $C_{n,t}=C_{n,sm}=f_{m}^{-1}(C_{n-1,s})\cap R_{m}$ for some
$m$ and then $F(C_{n,t})\subset C_{n-1,s}$. By this construction,
$z_{n}=f_{m_{n}}^{-1}\circ\dots\circ f_{m_{1}}^{-1}(z_{0})=\gamma_{n}(z_{0})$,
with $z_{0}\in\bigcap\overline{F^{n}(C_{n,t})}\subset\partial R$
and $\gamma_{n}=f_{m_{n}}^{-1}\circ\dots\circ f_{m_{1}}^{-1}\in\Gamma_{F}$.

Suppose that $\gamma_{i}=\gamma_{j}$ for some $i<j$. By construction,
$\gamma_{j}=f_{m_{j}}^{-1}\circ\dots\circ f_{m_{i+1}}^{-1}\circ f_{m_{i}}^{-1}\circ\dots\circ f_{m_{1}}^{-1}=f_{m_{j}}^{-1}\circ\dots\circ f_{m_{i+1}}^{-1}\circ\gamma_{i}$.
In other side, $z_{i}=\gamma_{i}(z_{0})=\gamma_{j}(z_{0})=z_{j}$.
Then 
\[
\begin{array}{lclcl}
z_{j+1} & = & f_{m_{i+1}}^{-1}\circ\gamma_{i}(z_{0}) & = & z_{i+1},\\
z_{j+2} & = & f_{m_{i+2}}^{-1}(z_{j+1}) & = & z_{i+2},\\
\dots\\
z_{j+j-i} & = & f_{m_{j}}^{-1}(z_{j-1}) & = & z_{j}=z_{i}.
\end{array}
\]
Therefore, the sequence $z_{0},z_{1},\dots,z_{i},\dots,z_{j},z_{i},\dots,z_{j},\dots\nrightarrow z\in L$,
contradicting the hypothesis.

In conclusion, the sequence $\gamma_{n}(z_{0})$ converging to $z\in L$
is constructed with distinct elements $\gamma_{n}\in\Gamma_{F}$ applied
to $z_{0}\in\bigcap_{n\geq0}\overline{F^{n}(C_{n,t})}\subset\partial R\subset\Omega(\Gamma_{F})$.
Notice that $\bigcap_{n\geq0}\overline{F^{n}(C_{n,t})}\neq\emptyset$
because $\overline{F^{n}(C_{n,t})}$ is a sequence of nested closed
sets. Finally, using Proposition \prettyref{prop:AlphaLimit1}, $\alpha(F)\subset L\subset\Lambda(\Gamma_{F})$.

Now let $z\in L$ but $z\notin\alpha(F)$. Then $z\in\overline{F^{-n}(\partial R)}\subset \mathcal{PD}_{n}(F)$
for some $n$. In consequence, $F^n(z)=\gamma(z)\in\partial R$ for
some $\gamma\in\Gamma_{F}$. Therefore $\emptyset\neq\left(\gamma\Lambda(\Gamma_{F})\right)\cap\partial R\subset\Lambda(\Gamma_{F})\cap\partial R$,
contradicting the hypothesis.\end{proof}

Recall that the $\omega$-limit of a point $z_{0}$ under the map
$F$ is the set of accumulation points of the orbit $\left\{ F^{n}(z_{0})\right\} $
and is denoted by $\omega(z_{0},F)$. Regarding the limit set of the
Kleinian group, we have the next

\begin{theo}\label{thm:Omega-Limit}$\omega(z,F)\subset\Lambda(\Gamma_{F})$
for all $z\in\CHat$.\end{theo}

\begin{proof}$\omega(z,F)$ is the set of accumulation points of
the orbit $\left\{ F^{n}(z)\right\} $. But \linebreak{}
$F^{n}=\gamma_{n}\in\Gamma_{F}$, then $\left\{ F^{n}(z)\right\} \subset\Gamma_{F}z$.
In the other hand, $\Lambda(\Gamma_{F})$ is the set of accumulation
points of $\Gamma_{F}z$, then $\omega(z,F)\subset\Lambda(\Gamma_{F})$.\end{proof}

\section{Deformations and Stability\label{sec:Deformations-and-Stability}}

The parameter space of piecewise conformal maps depends on the maps
\linebreak{}
$F\vert_{R_{m}}=\gamma_{m}\in PSL(2,\C)$ and the elements $R_{m}$
of the partition in the sphere.


In our case it is enough to consider the discontinuity set $\partial R=\bigcup_{m=1}^{M}{\partial}R_{m}$
as a compact subset of $\CHat$. So if the partition runs from
$m=1$ to $m=M$, the parameter space is a subspace of
\[
\mathcal{T}=\overbrace{PSL(2,\C)\times\dots\times PSL(2,\C)}^{M\,\mathrm{times}}\times\mathcal{H}(\CHat)
\]

with the product topology, where $\mathcal{H}(\CHat)$ is as in the Note on section 3.

We can ask about stability of $F$ through deformations in the parameter
space $\mathcal{T}$. First we will consider deformations of $\partial R$
and later on deformations of the coefficients on the component functions
$\gamma_{m}$ of $F$.

In order to be clear, along this section we consider $F$ to be defined
in only two simply connected regions $R_{1}$ and $R_{2}$, being
$\partial R$ one simple closed curve. Let $f=F|_{R_{1}}$ and $g=F|_{R_{2}}$.
To begin, let us fix $f$ and $g$, and perturb $\partial R$ to obtain
$\partial R'$. Now, $\partial R'$ bounds two regions homeomorphic
to discs $R'_{1}$ and $R'_{2}$ and we define $F'$ by $F'|_{R'_{1}}=f$,
$F'|_{R'_{2}}=g$. Notice $\Gamma_{F}=\Gamma_{F'}=\left\langle f,g\right\rangle $,
because $f$ and $g$ have been fixed.

Let $\mathcal{C}(\CHat)$ the subspace of $\mathcal{H}(\CHat)$
consisting of all compact subsets of the sphere homeomorphic to a
circle. Observe that for each $n\in{\N}$ and $f,g\in PSL(2,\C)$,
there are a natural maps $\Psi_{f,g,n}:\mathcal{C}(\CHat)\mapsto\mathcal{H}(\CHat)$
that assigns to $\partial R$ the \emph{n}th-level pre-discontinuity set $\mathcal{PD}_{n}(F)$
of the piecewise map $F$, and let us denote by $\Psi_{f,g}$ the
map from $\partial R$ to the pre-discontinuity set $\mathcal{PD}(F)$.

\begin{theo}\label{thm:ContDefSpidN}For a fixed pair $f,g$ in $PSL(2,\C)$,
the map $\Psi_{f,g,n}$ is continuous, for each $n\in{\N}$.\end{theo}

\begin{proof}We prove continuity of $\Psi_{f,g,n}$ using the sequence
convergence criterion. Let be $C_{k}\in\mathcal{C}(\CHat)$ a sequence
convergent to $\partial R$. Let be $D_{k}$ and $D$ the closure
of interior sets of $C_{k}$ and $\partial R$, respectively, and
$E_{k}$ and $E$ the closure of exterior sets of $C_{k}$ and $\partial R$,
respectively. Particularly, we have $D=\overline{R_{1}}$ and $E=\overline{R_{2}}$.
Because of \prettyref{lem:HausConvJordanCurve} (see \prettyref{sec:Technical-lemmas}),
$D_{k}\rightarrow D$ and $E_{k}\rightarrow E$.

Using \prettyref{lem:HausConvHom}, we have $f^{-1}(C_{k})\rightarrow f^{-1}(\partial R)$
and $g^{-1}(C_{k})\rightarrow g^{-1}(\partial R)$, because $f$ and
$g$ are M\"obius transformations.

Because of \prettyref{lem:HausConvInter}
\[
\begin{array}{ccc}
f^{-1}(C_{k})\,\cap\,D_{k} & \rightarrow & f^{-1}(\partial R)\,\cap\,D\,-\,Y_{f}\end{array}
\]
and
\[
\begin{array}{ccc}
g^{-1}(C_{k})\,\cap\,E_{k} & \rightarrow & g^{-1}(\partial R)\,\cap\,E\,-\,Y_{g}\end{array}
\]
where $Y_{f}$ and $Y_{g}$ are the respective isolated points sets.

Using \prettyref{lem:HausConvUnion} we have
\[
\begin{array}{cc}
\big(f^{-1}(C_{k})\,\cap\,D_{k}\big)\,\cup\,\big(g^{-1}(C_{k})\,\cap\,E_{k}\big)\,\cup\,C_{k} & \rightarrow\\
\big(f^{-1}(\partial R)\,\cap\,D\,-\,Y_{f}\big)\,\cup\,\big(g^{-1}(\partial R)\,\cap\,E\,-\,Y_{g}\big)\,\cup\,\partial R
\end{array}
\]

Let be $F_{k}$ piecewise transformations such that $F_{k}|_{D_{k}}=f$
and $F_{k}|_{D_{k}^{c}}=g$. Then 
\[
\big(f^{-1}(C_{k})\cap D_{k}\big)\cup\big(g^{-1}(C_{k})\cap E_{k}\big)\cup C_{k}=F^{-1}(C_{k})\cup C_{k}=\mathcal{PD}_{1}(F_{k}).
\]

In the other hand 
\[
\begin{array}{cc}
\big(f^{-1}(\partial R)\cap D-Y_{f}\big)\cup\big(g^{-1}(\partial R)\cap E-Y_{g}\big)\cup\partial R & =\\
\big(f^{-1}(\partial R)\cap D\big)\cup\big(g^{-1}(\partial R)\cap E\big)\cup\partial R & =\\
F^{-1}(\partial R)\cup\partial R & =\\
\mathcal{PD}_{1}(F),
\end{array}
\]
because $Y_{f},\,Y_{g}\subset\partial R$.

Finally, we have shown that $\mathcal{PD}_{1}(F_{k})\rightarrow \mathcal{PD}_{1}(F)$.

Now suppose that $\mathcal{PD}_{n-1}(F_{k})\rightarrow \mathcal{PD}_{n-1}(F)$. Then,
with an analogous argument to previous one using Lemmas \ref{lem:HausConvHom},
\ref{lem:HausConvInter} and \ref{lem:HausConvUnion}, but with $\mathcal{PD}_{n-1}(F_{k})$
instead of $C_{k}$, we can demonstrate that 
\[
F^{-1}\big(\mathcal{PD}_{n-1}(F_{k})\big)\cup C_{k}\rightarrow F^{-1}\big(\mathcal{PD}_{n-1}(F)\big)\cup\partial R
\]
But, by \prettyref{lem:F-1Spid}, we have that $F^{-1}\big(\mathcal{PD}_{n-1}(F_{k})\big)\cup C_{k}=\mathcal{PD}_{n}(F_{k})$
and\linebreak{}
 $F^{-1}\big(\mathcal{PD}_{n-1}(F)\big)\cup\partial R=$$\mathcal{PD}_{n}(F)$.\end{proof}

Even more

\begin{theo}\label{thm:ContDefSpid}For a fixed pair $f,g$ in $PSL(2,\C)$,
if $\partial R\cap\Lambda(\Gamma_{F})=\emptyset$, then the map $\Psi_{f,g}$
is continuous.\end{theo}

\begin{proof}As in the previous proof, let be $C_{k}\in\mathcal{C}(\CHat)$
a sequence convergent to $\partial R$ and $F_{k}$ the piecewise
conformal map defined by $f$ in $D_{k}$ the interior of $C_{k}$
and $g$ in $E_{k}$ the exterior of $C_{k}$. Also we assume that
$C_{k}\cap\Lambda(\Gamma_{F})=\emptyset$ for all $k$.

Let $z\in \mathcal{PD}(F)$. If $z\in \mathcal{PD}_{n}(F)$ for some $n$, then every
neighborhood of $z$, denoted $\mathcal{N}_{z}$, intersects infinitely
many $\mathcal{PD}(F_{k})$ because $\mathcal{PD}_{n}(F_{k})\rightarrow \mathcal{PD}_{n}(F)$.

If $z\in\alpha(F)$, every neighborhood $\mathcal{N}_{z}$ intersects
infinitely many $\overline{F^{-n}(\partial R)}$. For each $w\in\overline{F^{-n}(\partial R)}$
every neighborhood $\mathcal{N}_{w}\subset\mathcal{N}_{z}$ intersects
infinitely many $\mathcal{PD}_{n}(F_{k})$. Then, $\mathcal{N}_{z}$ intersects
infinitely many $\mathcal{PD}(F_{k})$.

Now let $z\in\CHat$ such that every neighborhood $\mathcal{N}_{z}$
intersects infinitely many $\mathcal{PD}(F_{k})$. If $z\in\Omega(\Gamma_{F})$,
$\mathcal{N}_{z}$ intersects infinitely many $\mathcal{PD}_{n}(F_{k})$ for
some fixed $n$, because $\mathcal{PD}_{n}(F_{k})\subset\Omega(\Gamma_{F})$
for all $n$ and $\underset{n\rightarrow\infty}{\lim}\overline{F^{-n}(\partial R)}\subset\Lambda(\Gamma_{F})$
by \prettyref{thm:Alpha-Limit}. Then, $z\in \mathcal{PD}_{n}(F)\subset \mathcal{PD}(F)$.

If $z\in\Lambda(\Gamma_{F})$, $\mathcal{N}_{z}$ can not intersect
infinitely many $\mathcal{PD}_{n}(F_{k})$ for some fixed $n$, because that
implies $z\in \mathcal{PD}_{n}(F)\subset\Omega(\Gamma_{F})$. Then, \textbf{$\mathcal{N}_{z}$}
must intersect sets $\overline{F_{k}^{-n_{k}}(\partial R)}\subset \mathcal{PD}(F_{k})$
with an increasing sequence $n_{k}$. As $\mathcal{PD}_{n_{k}}(F_{k})\rightarrow \mathcal{PD}_{n_{k}}(F)$
for each $n_{k}$, $\mathcal{N}_{z}$ intersects infinitely many $\overline{F^{-n_{k}}(\partial R)}$,
and we conclude that\linebreak{}
 $z\in\underset{n\rightarrow\infty}{\lim}\overline{F^{-n}(\partial R)}=\alpha(F)$
using \prettyref{thm:Alpha-Limit}.\end{proof}

In this way, if $\partial R\cap\Lambda(\Gamma_{F})=\emptyset$, we
ensure certain stability of $F$ (fixing $f$ and $g$) through continuous
deformations of $\partial R$ (see figure \ref{fig:spidstable}).
In the opposite case, if $\partial R\cap\Lambda(\Gamma_{F})\neq\emptyset$,
we found unrelated dynamics of $F$ and $F'$, because the unstability
of $\Lambda(\Gamma_{F})$ carried by $F^{-n}(\partial R$) along $\mathcal{PD}(F)$
(see figure \ref{fig:spidunstable}).

Let be $X_{\partial R}=\left\{ F:\CHat\rightarrow\CHat:\,F|_{R_{1}},\,F|_{R_{2}}\in PSL(2,\C)\right\} $
the space of piecewise conformal transformations defined by a fix
$\partial R$, a simple closed curve. This space is homeomorphic to
$PSL(2,\C)\times PSL(2,\C)$. 

A piecewise conformal transformation $F\in X_{\partial R}$ is \textbf{structurally
stable} if exists $\mathcal{N}_{F}$ neighborhood of $F$ in $X_{\partial R}$
such that for all $F'\in\mathcal{N}_{F}$ exists $h:\CHat\rightarrow\CHat$
homeomorphism with $F'\circ h=h\circ F$.

The following result can ensure structural stability on the particular
class of transformations which have a classic or non-classic Schottky group as associated
Kleinian group.

\begin{theo}\label{thm:StrucStabSchottky}If $\Gamma_{F}$ is a Schottky
group and $\partial R$ is contained in a fundamental region of $\Gamma_{F}$,
then $F$ is structurally stable in $X_{\partial R}$.\end{theo}

\begin{proof}Being $\Gamma_{F}=\left\langle f,g\right\rangle $ a
Schottky group, exists $C_{1},C_{2},C_{3},C_{4}$ simple closed curves such that
$f$ maps $C_{1}$ onto $C_{2}$ reversing orientation, that is, the
domain interior of $C_{1}$ is mapped to the domain exterior of $C_{2}$.
Analogously for $g$ mapping $C_{3}$ to $C_{4}$. Those closed curves bound
a common region $\mathcal{R}$ (called fundamental region), and can
be chosen to contain $\partial R$.

Marked Schottky groups with two generators are a domain in $\C^{3}$, then the inverse
image from projection to marked Schottky groups are a domain in $PSL(2,\C)\times PSL(2,\C)$
(see note below). In this way, we can choose a neighborhood $\mathcal{N}_{F}$
in $X_{\partial R}$ such that for all $F'\in\mathcal{N}_{F}$ the
group $\Gamma_{F'}=\left\langle f',g'\right\rangle $ is a Schottky
group and $\partial R$ is into the fundamental region $\mathcal{R}'$
of $\Gamma_{F'}$.

All sets $\gamma(\partial R)$ with $\gamma\in\Gamma_{F'}$ are distinct
and do not intersect each other, because $\partial R\subset\mathcal{R}'\subset\Omega(\Gamma_{F'})$,
$f'(\mathcal{R}')\cap\mathcal{R}'=\emptyset$, $g'(\mathcal{R}')\cap\mathcal{R}'=\emptyset$
and $\Gamma_{F'}$ is a free group (see \cite{Mas}). Then for all
$F'\in\mathcal{N}_{F}$, $\mathcal{PD}(F')$ is union of distinct non-intersecting
simple closed curves plus the associated $\alpha(F')\subset\Lambda(\Gamma_{F'})$.

We construct $\varphi:\mathcal{N}_{F}\times E\rightarrow\CHat$,
an holomorphic motion of \linebreak{}
$E=\bigcup_{n\geq0}F^{-n}(\partial R)\subset \mathcal{PD}(F)$, as follow.
For $\lambda=(\lambda_{1},\dots,\lambda_{6})$ associated to $F'\in\mathcal{N}_{F}$
and $z\in F^{-n}(\partial R)\subset E$, define $\varphi(\lambda,z)=F'^{-n}\circ F^{n}(z)$.
Observe that $F^{n}(z)\in\partial R$ and $F'^{-n}\circ F^{n}(z)\in F'^{-n}(\partial R)$.
Each function $\varphi_{\lambda}=\varphi(\lambda,\_)$ is an injection
on $\CHat$ because $\varphi_{\lambda}$ is defined by a M\"obius
transformation in each closed curve forming the set $F^{-n}(\partial R)$.
$F'^{-n}$ is composition of M\"obius transformations $f'^{-1}$ and
$g'^{-1}$, being rational functions of $\lambda_{1},\lambda_{2},\lambda_{3}$
and $\lambda_{4},\lambda_{5},\lambda_{6}$ respectively, then $\varphi(\lambda,z_{0})$
is an holomorphic function of $\lambda$ for each fixed $z_{0}$.
If $\lambda_{0}$ is the element associated to $F$, is clear that
$\varphi(\lambda_{0},z)=z$.

Using the extended $\lambda$-lemma (see \cite{ST}), the holomorphic
motion $\varphi$ has an extension to an holomorphic motion $\tilde{\varphi}$
of $\overline{E}=\mathcal{PD}(F)$. Even more, for each $\lambda\in\mathcal{N}_{F}$,
$\tilde{\varphi}_{\lambda}$ extends to a quasiconformal homeomorphism
$h_{\lambda}:\CHat\rightarrow\CHat$. By construction, $h_{\lambda}$
conjugates $F$ with $F'$.\end{proof}

\begin{note}Recall that for two finitely generated Kleinian groups
$\Gamma_{1}=\left\langle f_{1},\dots,f_{n}\right\rangle $ and $\Gamma_{2}=\left\langle g_{1},\dots,g_{n}\right\rangle $,
it can be defined an equivalence $\Gamma_{1}\sim\Gamma_{2}$ if exists
a M\"obius transformation $h$ such that $h\circ f_{i}\circ h^{-1}=g_{i}$.
If $S\subset PSL(2,\C)\times PSL(2,\C)$ is the set of Schottky groups
of two generators, then $S/\sim$ is the set
of marked Schottky groups and is a domain in $\C^{3}$ (see \cite{Bers})
and therefore $S$ is also a domain.\end{note}

\newpage{}

\section{Appendix: Examples\label{sec:Examples}}

In all images, periodic points are colored in red and the pre-discontinuity set
in black. Other coloring is used to distinguish Fatou components.

\begin{figure}[H]

\begin{overpic}[width=0.52\textwidth]{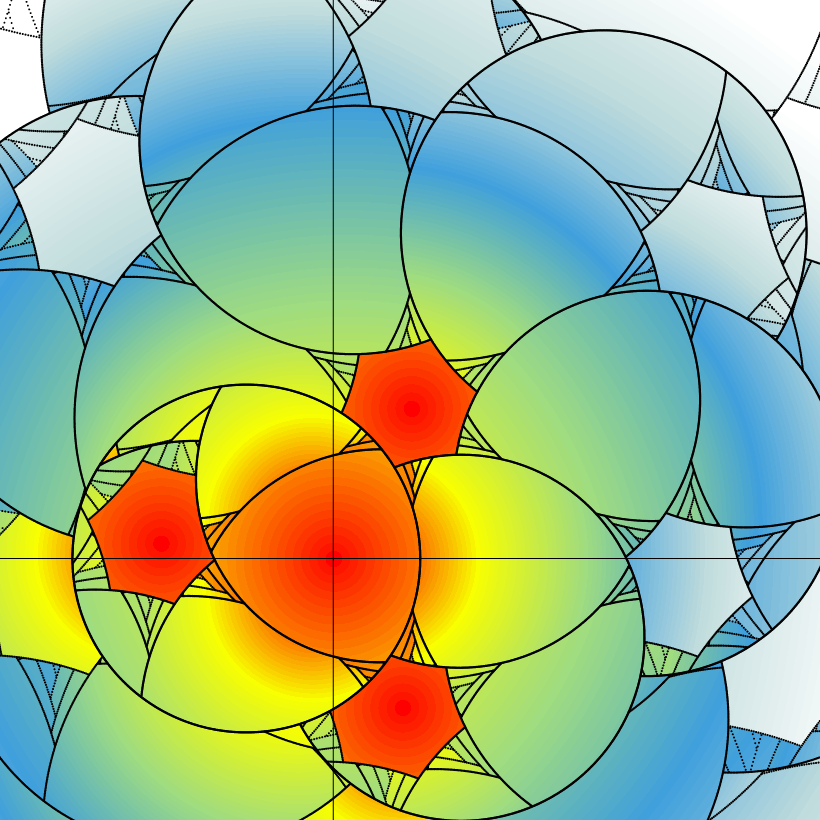} 
\put (35,35) {\huge$U$}
\put (45,50) {\huge$V$}
\put (15,32) {\huge$W$}
\put (42,10) {$F(W)$}
\end{overpic}\medskip{}

\caption{\label{fig:attr}\textbf{Fatou components: Attractive basins.} \protect \\
$F|_{R}(z)=\lambda z$ and $F|_{R^{c}}(z)=\lambda(1-z)$, where \protect \linebreak{}
$R=\left\{ z:\,|-\frac{1}{2}-z|<1\right\} $ and $\lambda=0.95e^{\frac{2}{3}\pi i}$.
Attractive basins $U$, $V$ and $W$, such that $F(U)\subset U$,
$F(V)\subset V$ and $F^{2}(W)\subset W$.}
\end{figure}

\begin{figure}[H]
\begin{overpic}[width=0.52\textwidth]{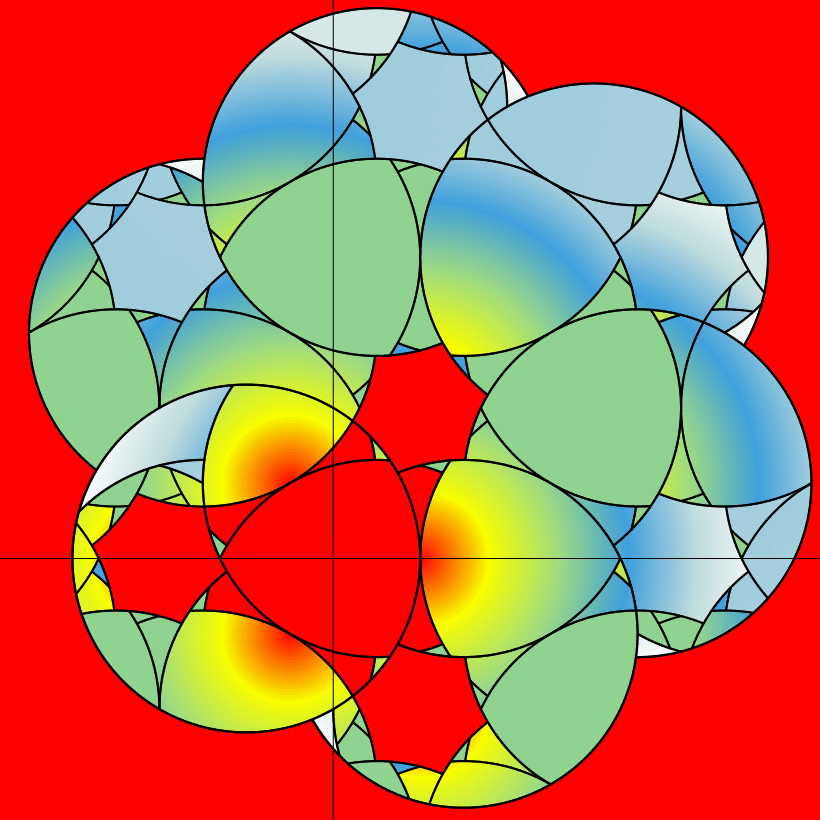} 
\put (35,32) {\huge$U$}
\put (47,48) {\huge$V$}
\put (15,30) {\huge$W$}
\put (46,10) {$F(W)$}
\put (10,90) {\huge$E$}
\end{overpic}\medskip{}

\caption{\label{fig:rot}\textbf{Fatou components: Rotation domains.} \protect \\
$F|_{R}(z)=\lambda z$ and $F|_{R^{c}}(z)=\lambda(1-z)$, where \protect \linebreak{}
$R=\left\{ z:\,|-\frac{1}{2}-z|<1\right\} $ and $\lambda=e^{\frac{2}{3}\pi i}$.
Simply connected rotation domains $U$, $V$ , $W$ and $E$ containing
a fixed point, such that $F(U)=U$, $F(V)=V$, $F^{2}(W)=W$ and $F(E)=E$.}
\end{figure}

\begin{figure}[H]
\begin{overpic}[width=0.52\textwidth]{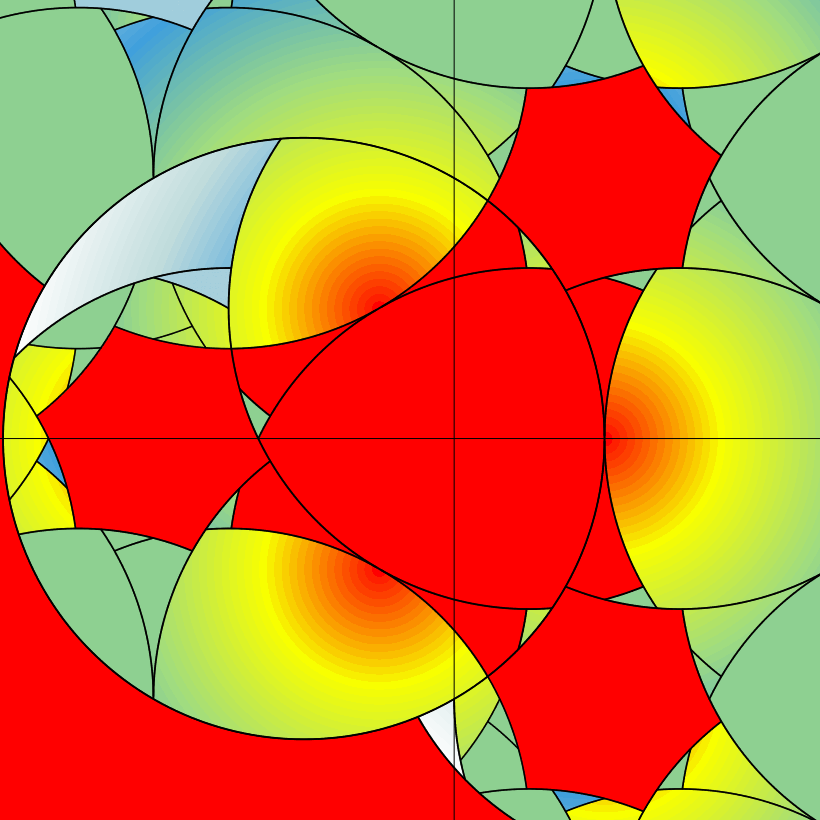} 
\put (30,54) {$U$}
\put (54,22) {$F(U)$}
\put (72,60) {$F^2(U)$}
\put (55,68) {$F^3(U)$}
\put (30,38) {$F^4(U)$}
\put (71,30) {$F^5(U)$}
\end{overpic}

\medskip{}
\caption{\label{fig:rotneutr}\textbf{Fatou components: Neutral domains.} \protect \\
$F|_{R}(z)=\lambda z$ and $F|_{R^{c}}(z)=\lambda(1-z)$, where \protect \linebreak{}
$R=\left\{ z:\,|-\frac{1}{2}-z|<1\right\} $ and $\lambda=e^{\frac{2}{3}\pi i}$.
Neutral domain $U$ such that $F^{6}|_{U}=Id$.}
\end{figure}

\begin{figure}[H]
\begin{overpic}[width=0.52\textwidth]{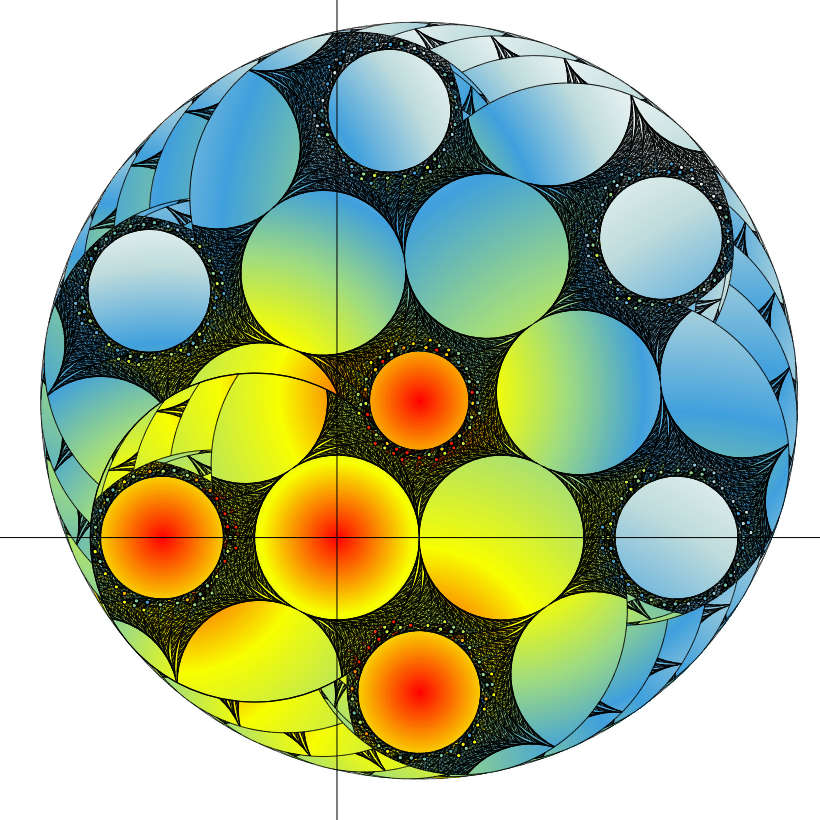} 
\put (35,35) {\huge$U$}
\put (47,50) {\huge$V$}
\put (15,32) {\huge$W$}
\put (46,10) {$F(W)$}
\put (10,90) {\huge$E$}
\end{overpic}

\medskip{}
\caption{\label{fig:rotirr}\textbf{Fatou components: Rotation domains.} \protect \\
$F|_{R}(z)=\lambda z$ and $F|_{R^{c}}(z)=\lambda(1-z)$, where \protect \linebreak{}
$R=\left\{ z:\,|-\frac{1}{2}-z|<1\right\} $ and $\lambda=e^{\alpha\pi i}$
with $\alpha$ an irrational number. Rotation domains $U$, $V$ ,
$W$ and $E$ containing a fixed point, such that $F(U)=U$, $F(V)=V$,
$F^{2}(W)=W$ and \protect \linebreak{}
$F(E)=E$.}
\end{figure}

\begin{figure}[H]
\begin{overpic}[width=0.52\textwidth]{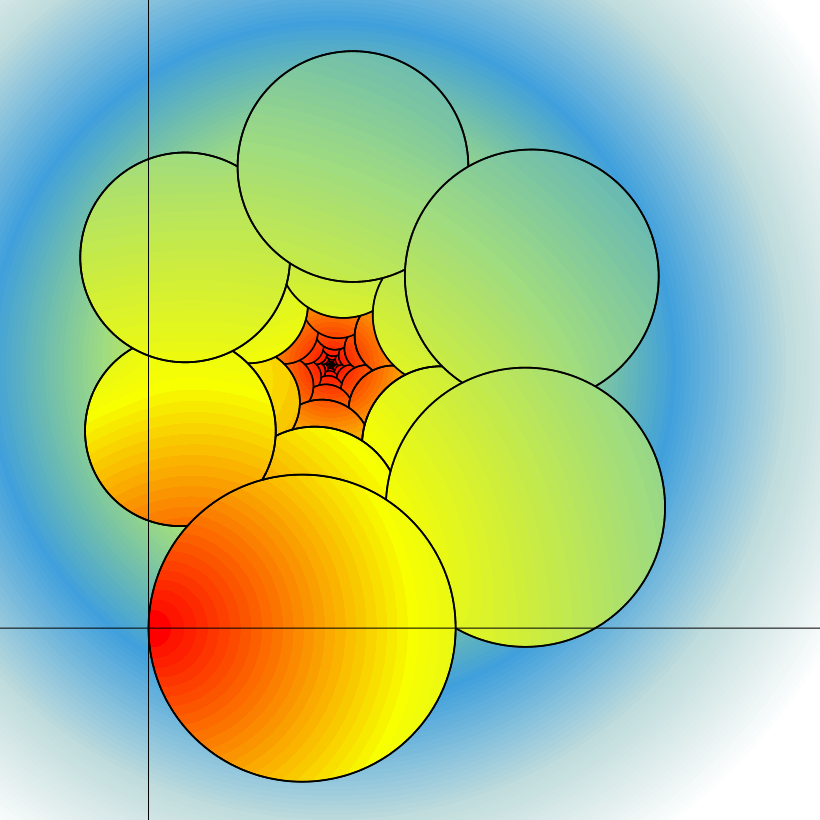} 
\put (30,30) {\Huge$U$}
\end{overpic}

\medskip{}
\caption{\label{fig:parab}\textbf{Fatou parabolic basin.} \protect \\
$F|_{R}(z)=\frac{z}{z+1}$ and $F|_{R^{c}}(z)=\lambda(1-z)$, where
\protect \linebreak{}
$R=\left\{ z:\,|-\frac{1}{2}-z|<\frac{1}{2}\right\} $ and $\lambda=1.1e^{\frac{2}{3}\pi i}$.
Parabolic basin $U$ with parabolic fixed point $0$ at $\partial U$,
such that $F(U)\subset U$. Parabolic basins with attractive point
in $\partial U$ behaves analogously.}
\end{figure}

\begin{figure}[H]
\begin{overpic}[width=0.52\textwidth]{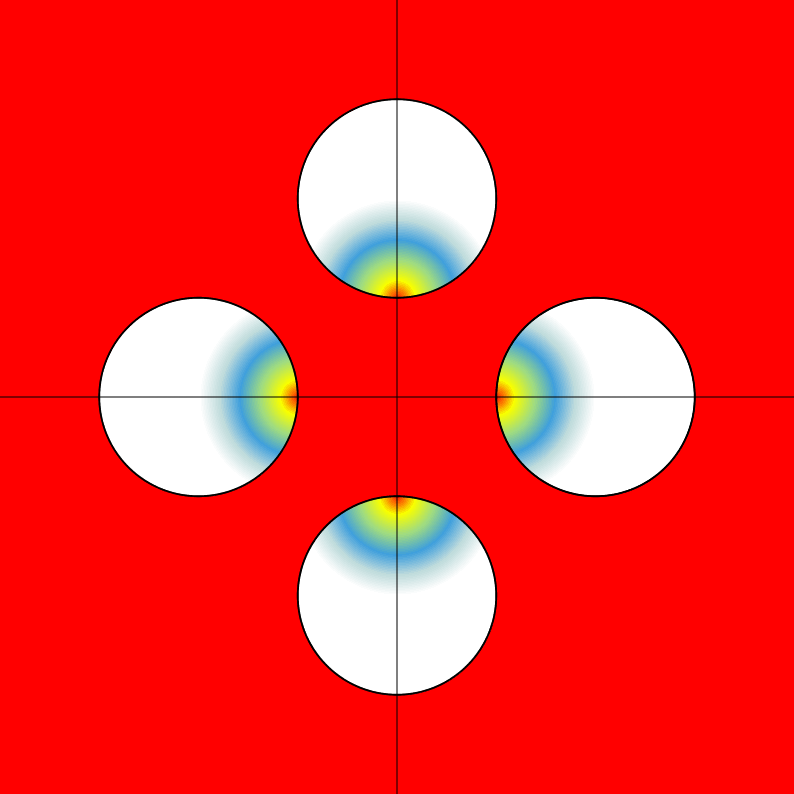} 
\put (20,80) {\Huge$U$}
\end{overpic}

\medskip{}
\caption{\label{fig:rot2fix}\textbf{Fatou components: Rotation domain.} \protect \\
$F|_{R}(z)=\lambda-\lambda z$ and $F|_{R^{c}}(z)=\lambda z$, where
\protect \linebreak{}
$R=\left\{ z:\,|1-z|<\frac{1}{2}\right\} $ and $\lambda=e^{\frac{1}{2}\pi i}$.
Rotation domain $U$ containig both fixed points of $F|_{U}$.}
\end{figure}

\begin{figure}[H]
\begin{overpic}[width=0.52\textwidth]{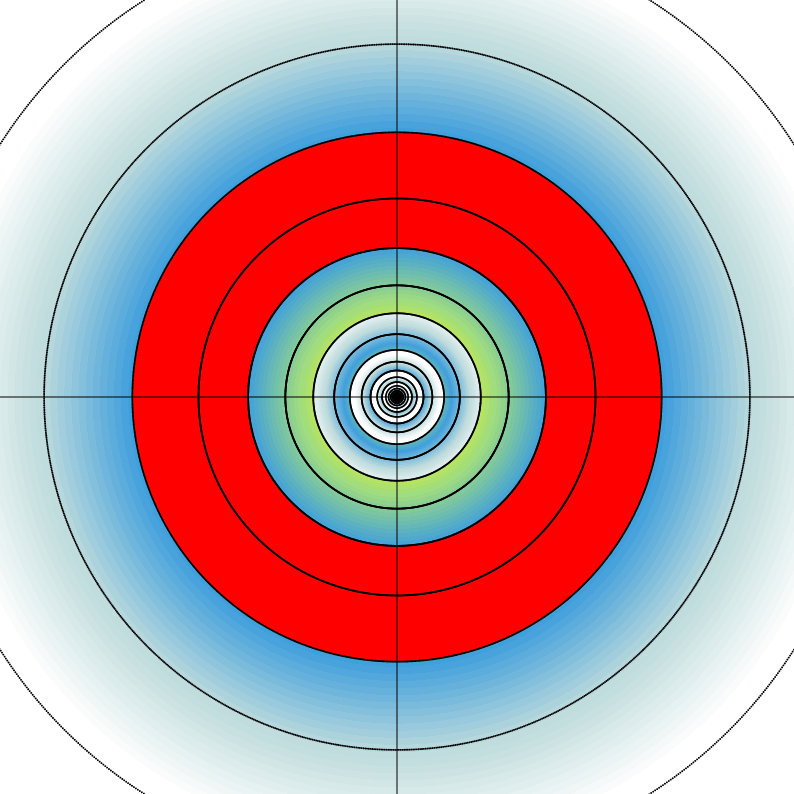} 
\put (40,75) {$U$}
\put (50,70) {$F(U)$}
\end{overpic}

\medskip{}
\caption{\label{fig:rotann}\textbf{Fatou components: Rotation domains.}\protect \\
$F|_{R}(z)=\frac{4}{3}\lambda z$ and $F|_{R^{c}}(z)=\frac{3}{4}\lambda z$,
where $R=\left\{ z:\,|z|<1\right\} $ and $\lambda=e^{\frac{1}{3}\pi i}$.
Rotation domain $U$ without fixed points, such that $F^{2}|_{U}(z)=e^{\frac{2}{3}\pi i}z$.}
\end{figure}

\begin{figure}[H]
\begin{overpic}[width=0.52\textwidth]{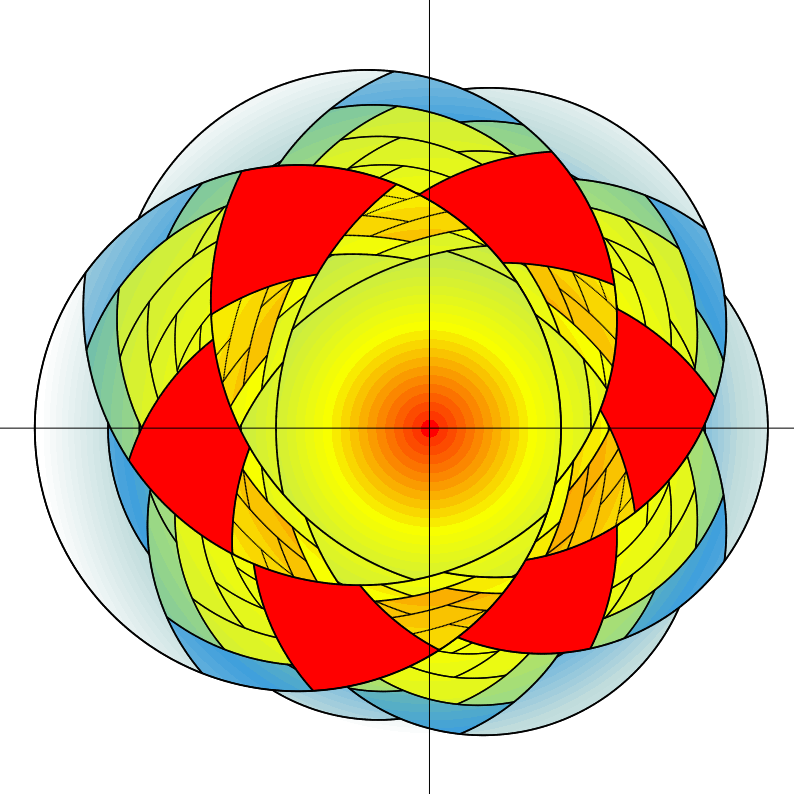} 
\put (35,70) {$U$}
\put (18,42) {$F(U)$}
\put (36,18) {$F^2(U)$}
\put (62,20) {$F^3(U)$}
\put (76,48) {$F^4(U)$}
\put (58,72) {$F^5(U)$}
\end{overpic}

\medskip{}
\caption{\label{fig:rotext}\textbf{Fatou neutral domains.}\protect \\
$F|_{R}(z)=\frac{95}{100}\lambda z$ and $F|_{R^{c}}(z)=\frac{100}{95}\lambda z$,
where \protect \linebreak{}
$R=\left\{ z:\,|-\frac{1}{2}-z|<1\right\} $ and $\lambda=e^{\frac{1}{3}\pi i}$.
Neutral domain $U$ such that $F^{6}|_{U}=Id$.}
\end{figure}

\begin{figure}[H]
\begin{overpic}[width=0.52\textwidth]{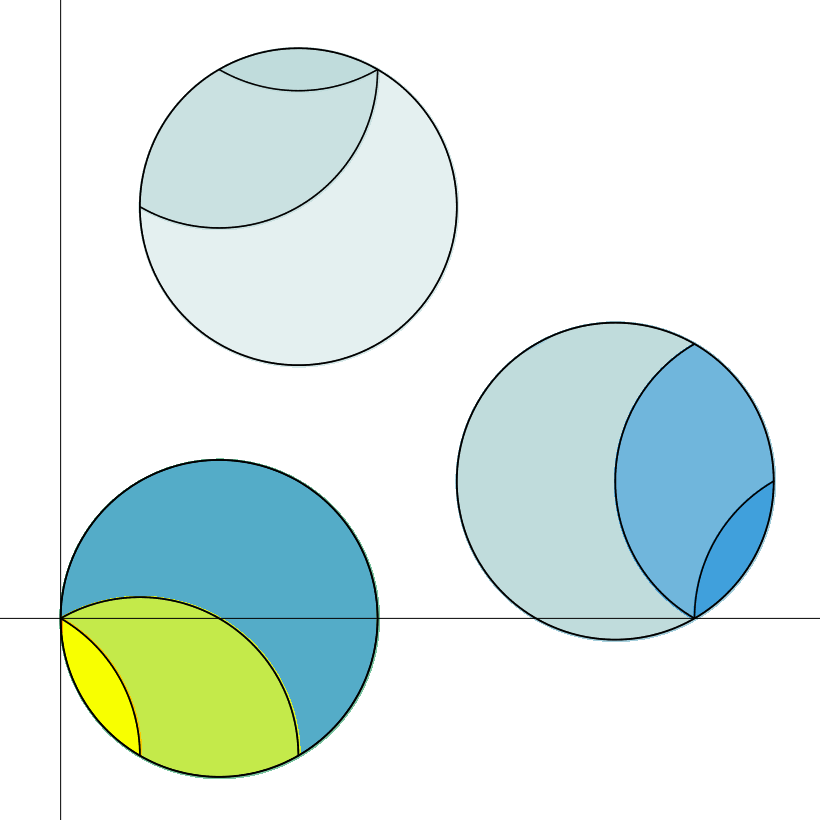}
\put (12,12) {$(0,0,0,1,1,\dots)$}
\put (20,20) {$(0,0,1,1,\dots)$}
\put (15,32) {$(0,1,1,\dots)$}
\put (90,30) {$(1,0,0,0,1,1,\dots)$}
\put (80,45) {$(1,0,0,1,1,\dots)$}
\put (70,55) {$(1,0,1,1,\dots)$}
\put (32,92) {$(1,1,0,0,0,1,1,\dots)$}
\put (22,80) {$(1,1,0,0,1,1,\dots)$}
\put (25,65) {$(1,1,0,1,1,\dots)$}
\put (5,50) {$(1,1,\dots)$}
\end{overpic}

\medskip{}
\caption{\label{fig:itin}\textbf{Itineraries.}\protect \\
$F|_{R}(z)=\lambda z$ and $F|_{R^{c}}(z)=\lambda(1-z)$, where $R=\left\{ z:\,|\frac{1}{4}-z|<\frac{1}{4}\right\} $
and $\lambda=e^{\frac{1}{3}\pi i}$. Each Fatou component is determined
by one itinerary.}
\end{figure}

\begin{figure}[H]
\begin{overpic}[width=0.52\textwidth]{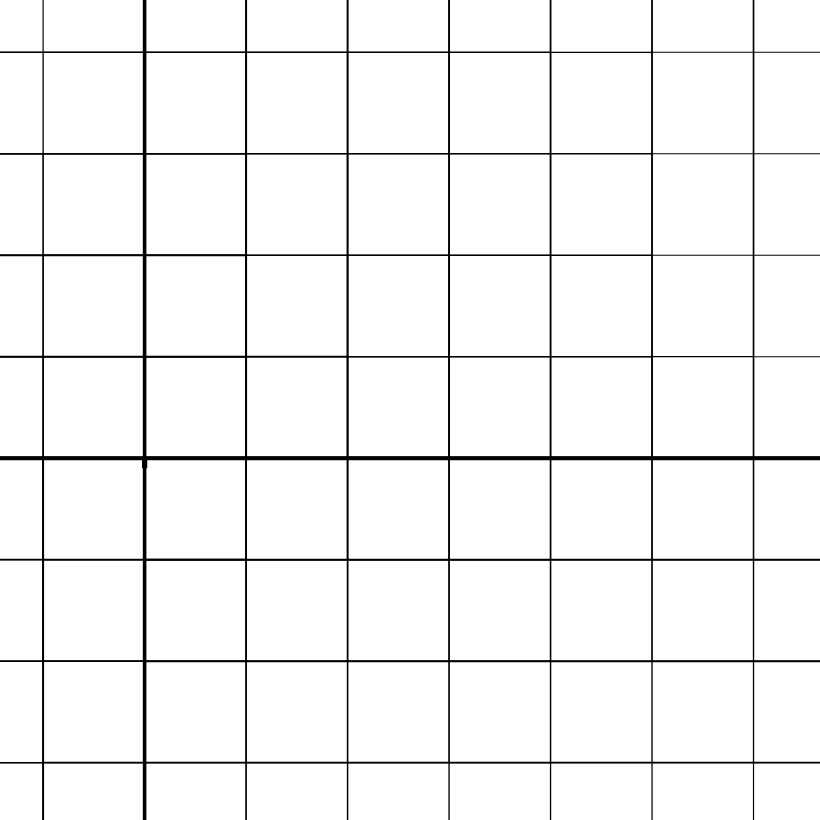}
\put (15,41) {\tiny$0$}
\put (19,49) {\tiny$U$}
\put (30,49) {\tiny$F(U)$}
\put (30,36) {\tiny$F^2(U)$}
\put (18,62) {\tiny$F^3(U)$}
\put (43,49) {\tiny$F^4(U)$}
\put (30,23) {\tiny$F^5(U)$}
\put (30,62) {\tiny$F^6(U)$}
\put (43,36) {\tiny$F^7(U)$}
\put (18,74) {\tiny$F^8(U)$}
\put (55,49) {\tiny$F^9(U)$}
\put (30,11) {\tiny$F^{10}(U)$}
\put (43,62) {\tiny$F^{11}(U)$}
\put (43,23) {\tiny$F^{12}(U)$}
\put (30,74) {\tiny$F^{13}(U)$}
\put (55,36) {\tiny$F^{14}(U)$}
\put (18,86) {\tiny$F^{15}(U)$}
\end{overpic}

\medskip{}
\caption{\label{fig:wander}\textbf{Wandering domains.}\protect \\
$F|_{R}(z)=iz$ and $F|_{R^{c}}(z)=-iz+1+i$, where $R=\left\{ z:\,Im(z)<0\right\} $.
First iterations in the orbit of component $U=(0,1)\times(0,1)$.}
\end{figure}

\begin{figure}[H]
\begin{centering}
\includegraphics[scale=0.2]{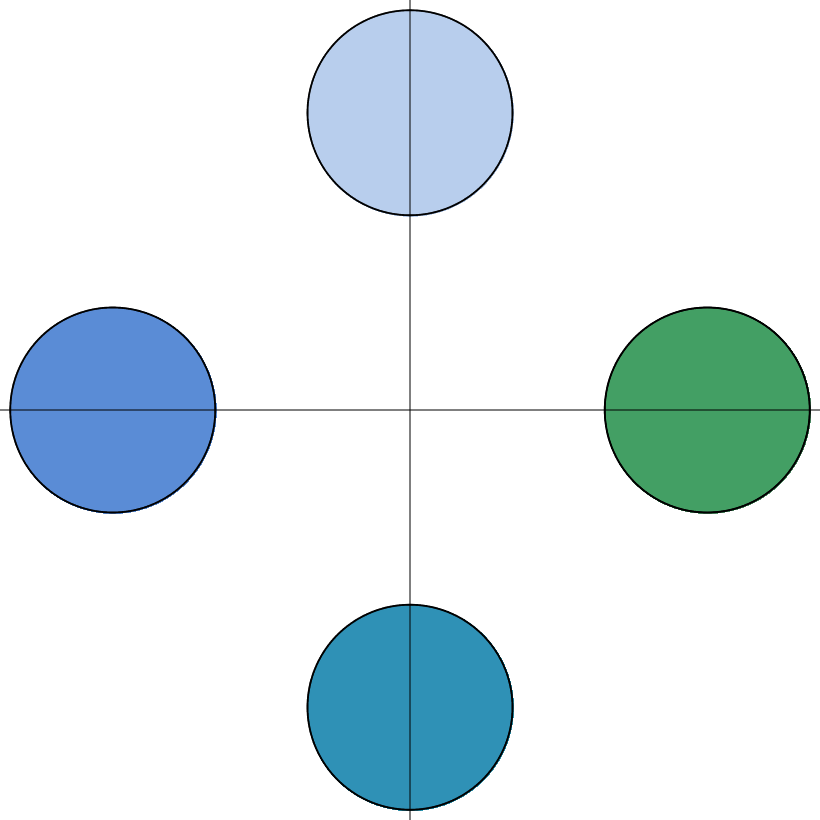}~~~~\includegraphics[scale=0.2]{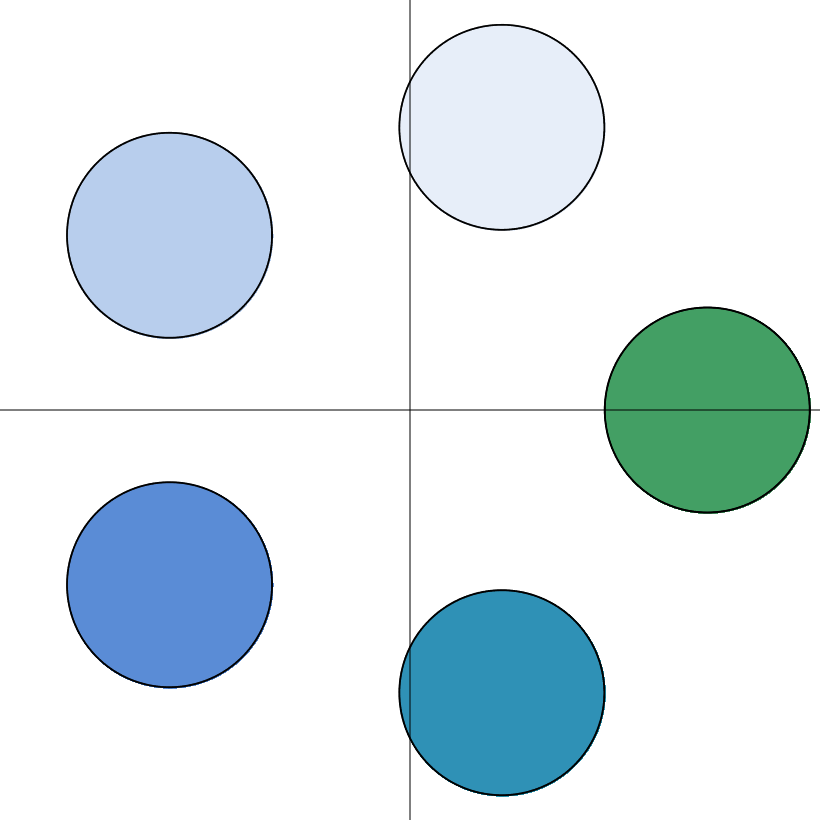}
\par\end{centering}
\medskip{}
\caption{\label{fig:conn}\textbf{Connectivity of Fatou components.}\protect \\
$F|_{R}(z)=2z$ and $F|_{R^{c}}(z)=\lambda z$, where $R=\left\{ z:\,|\frac{3}{2}-z|<\frac{1}{2}\right\} $
and $\lambda=e^{\frac{1}{2}\pi i}$ at left and $\lambda=e^{\frac{2}{5}\pi i}$
at right. The component outside circles is $4$-connected on left
and $5$-connected on right.}
\end{figure}

\begin{figure}[H]
\begin{centering}
\includegraphics[scale=0.2]{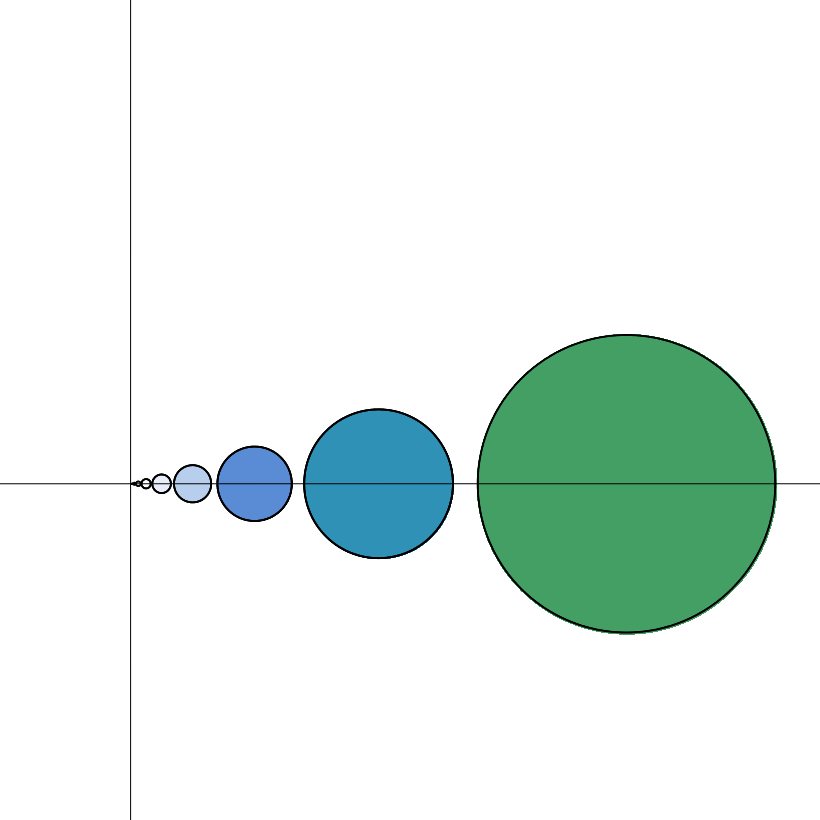}~~~~\includegraphics[scale=0.2]{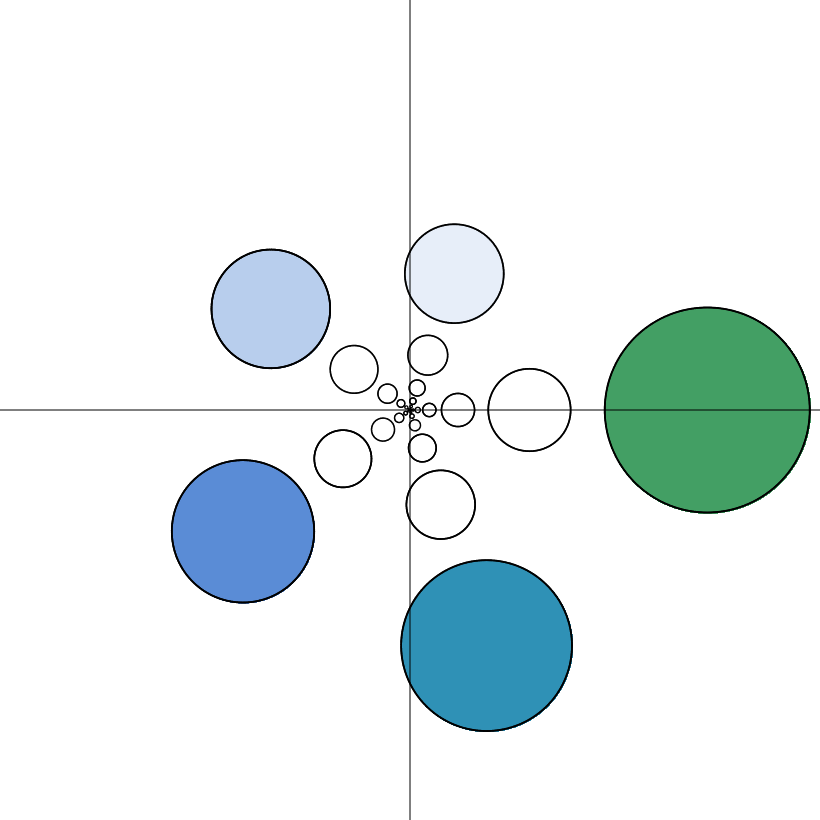}
\par\end{centering}
\medskip{}
\caption{\label{fig:conninfty}\textbf{Fatou components with $\infty$ connectivity.}\protect \\
At left, $F|_{R}(z)=z$ and $F|_{R^{c}}(z)=2z$, where\protect \linebreak{}
$R=\left\{ z:\,|1-z|<\frac{1}{3}\right\} $. At right, $F|_{R}(z)=2z$
and $F|_{R^{c}}(z)=\lambda z$, where $R=\left\{ z:\,|\frac{3}{2}-z|<\frac{1}{2}\right\} $
and $\lambda=1.2e^{\frac{2}{5}\pi i}$. In both cases, the component
outside circles is $\infty$-connected.}
\end{figure}

\begin{figure}[H]
\begin{centering}
\includegraphics[scale=0.2]{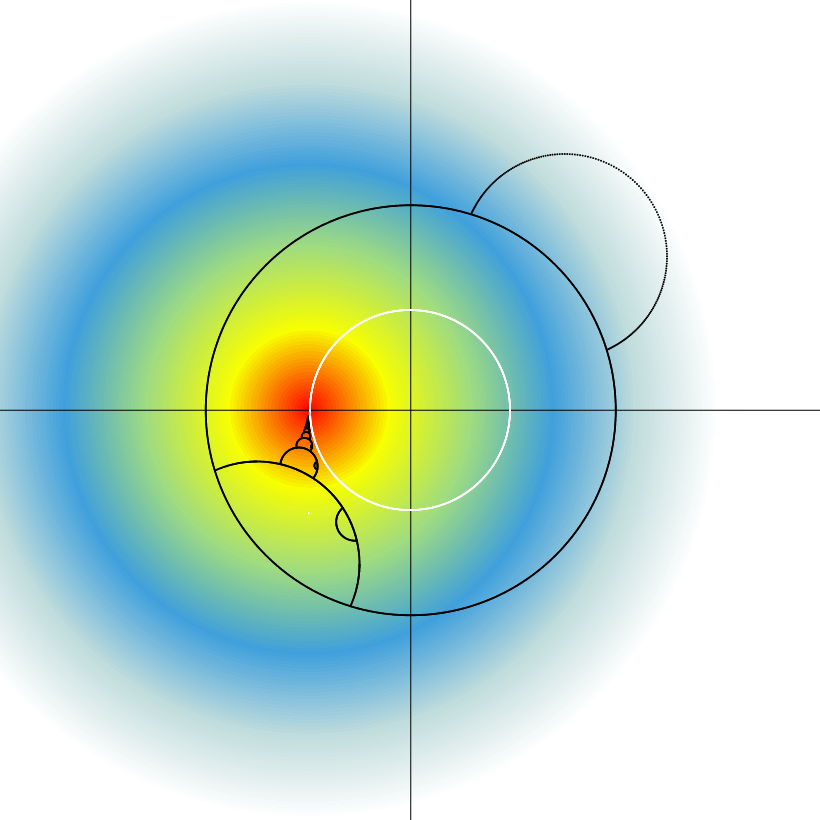}~~~~\includegraphics[scale=0.2]{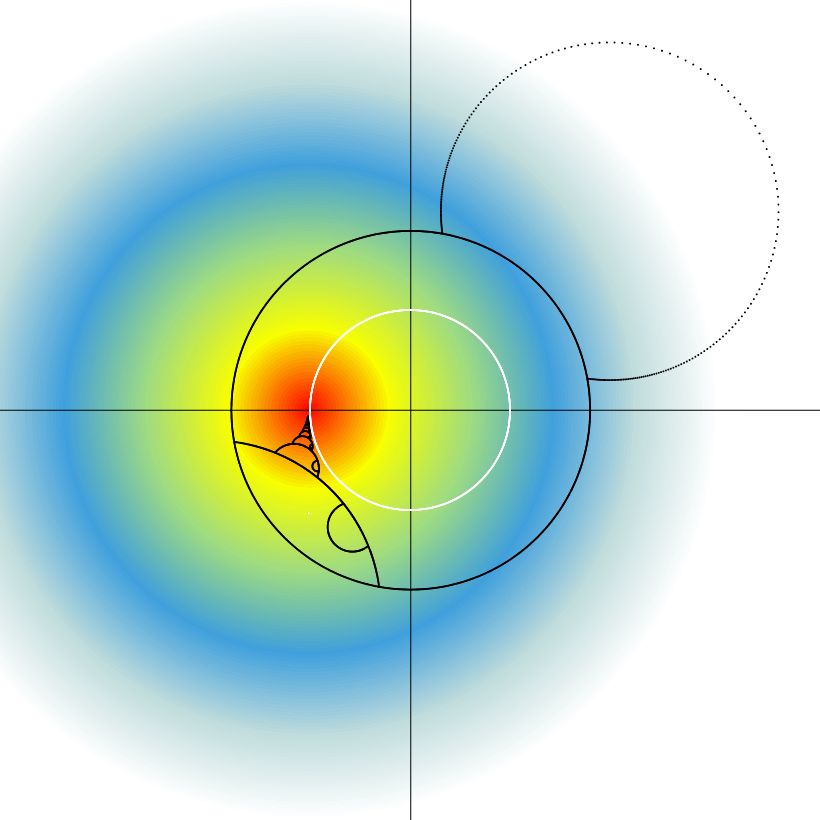}
\par\end{centering}
\medskip{}
\caption{\label{fig:spidstable}\textbf{Stability through deformations of $\partial R$.}\protect \\
$F|_{R}(z)=\frac{(1+i)z+i}{-i+(1-i)z}$ and $F|_{R^{c}}(z)=\frac{(1+i)z-i}{i+(1-i)z}$,
with \protect \linebreak{}
$R=\left\{ z:\,|z|<2\right\} $ at left and $R=\left\{ z:\,|z|<2-\varepsilon\right\} $
for some $\varepsilon>0$ at right. $\Gamma_{F}$ is a Fuchsian group
with $\Lambda(\Gamma_{F})=S^{1}$ (in white), then $\partial R\cap\Lambda(\Gamma_{F})=\emptyset$
in both cases. Therefore the map $\partial R\protect\mapsto \mathcal{PD}(F)$
is continuous.}
\end{figure}

\begin{figure}[H]
\begin{centering}
\includegraphics[scale=0.15]{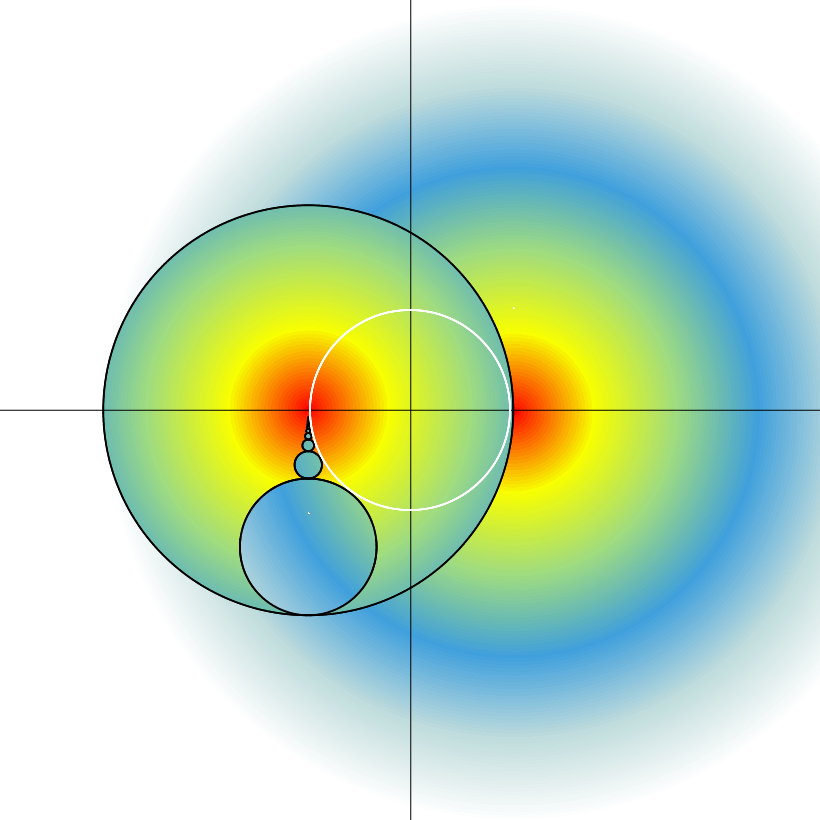}~~\includegraphics[scale=0.15]{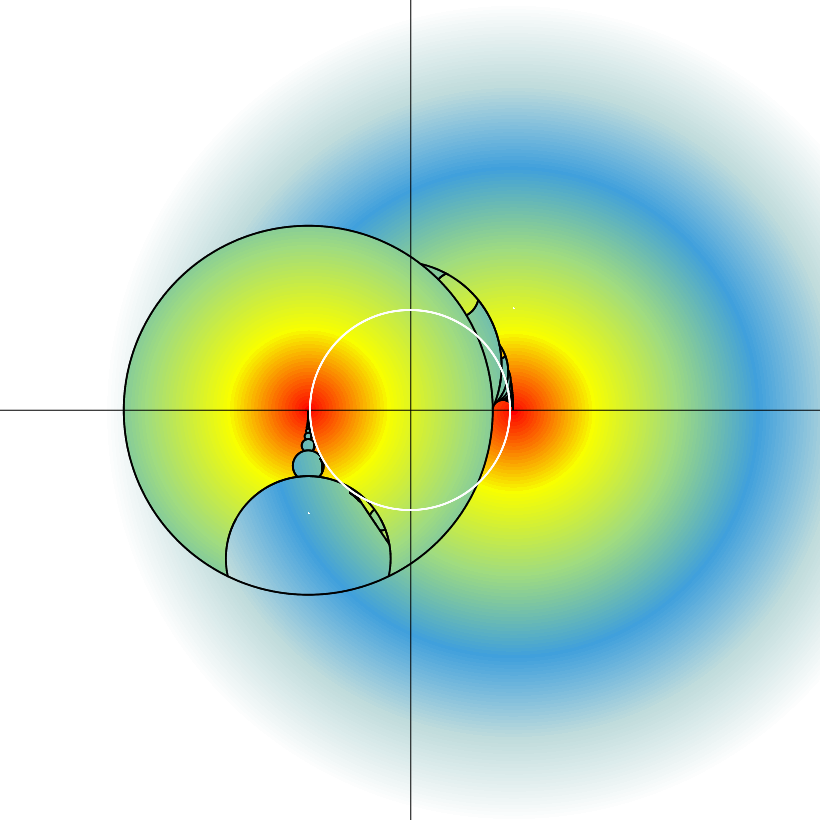}~~\includegraphics[scale=0.15]{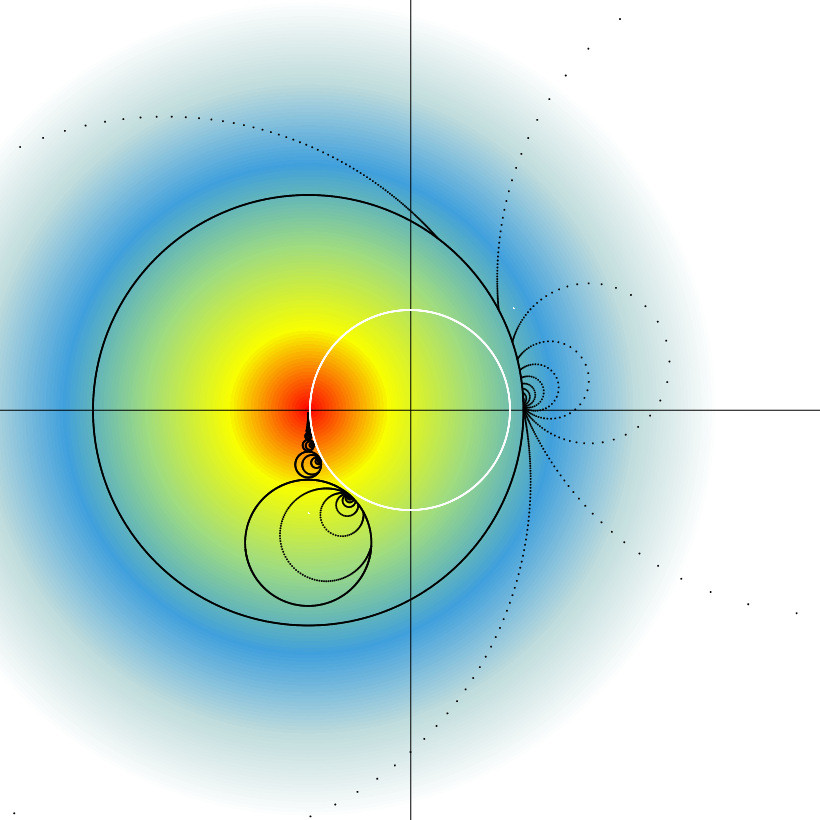}
\par\end{centering}
\medskip{}
\caption{\label{fig:spidunstable}\textbf{Unstability through deformations
of $\partial R$.}\protect \\
The same transformations from figure \ref{fig:spidstable}, but \protect \linebreak{}
$R=\left\{ z:\,|-\frac{3}{2}-z|<1\right\} $ at left, $R=\left\{ z:\,|-\frac{3}{2}-z|<1-\varepsilon\right\} $
at center and $R=\left\{ z:\,|-\frac{3}{2}-z|<1+\varepsilon\right\} $
at right, for some $\varepsilon>0$. Since $\Lambda(\Gamma_{F})=S^{1}$ (in white),
$\partial R\cap\Lambda(\Gamma_{F})\protect\neq\emptyset$. In those
cases, the map $\partial R\protect\mapsto \mathcal{PD}(F)$ is not continuous.
In other words, $F$ is unstable under deformations of $\partial R$.}
\end{figure}

\begin{figure}[H]
\begin{centering}
\includegraphics[scale=0.2]{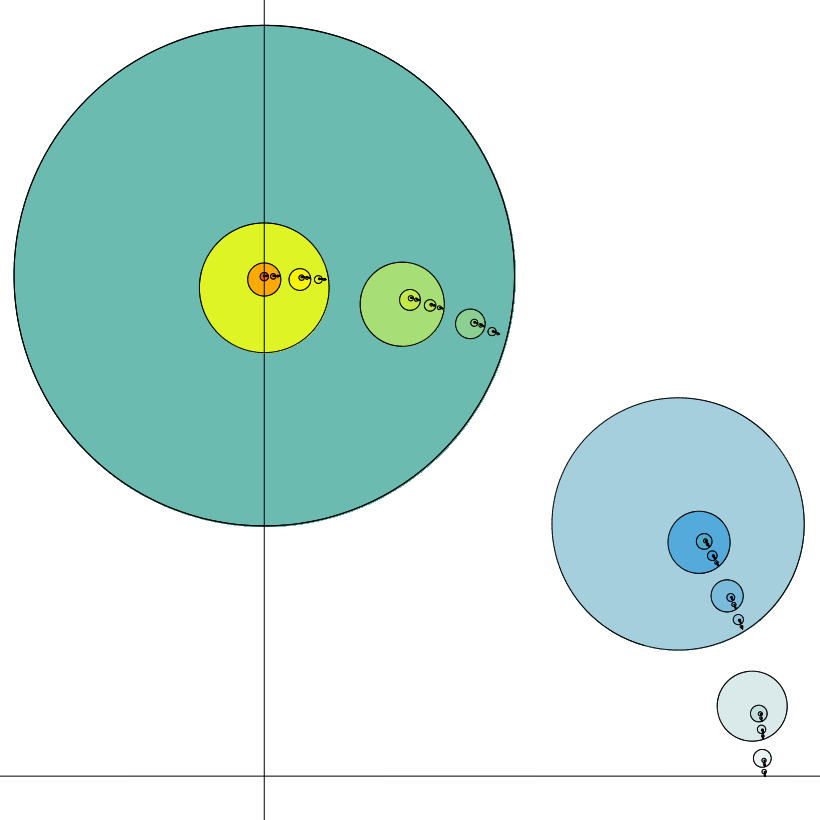}~~~~\includegraphics[scale=0.2]{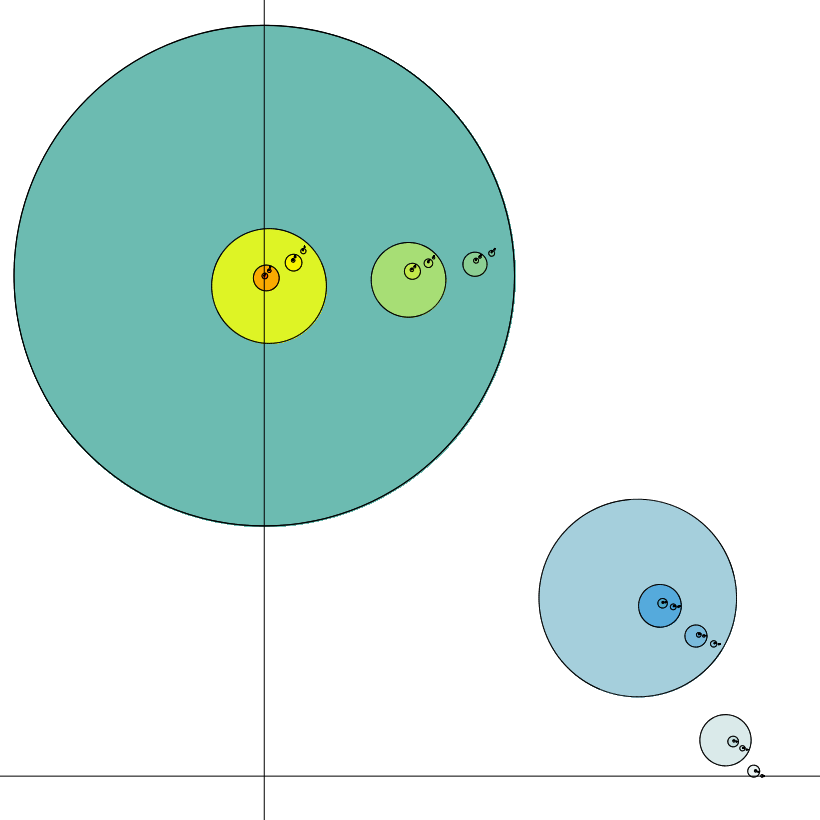}
\par\end{centering}
\medskip{}
\caption{\label{fig:schottky}\textbf{Structural stability.}\protect \\
$F|_{R}(z)=\frac{z-0.6i}{0.6i+z}$ and $F|_{R^{c}}(z)=\frac{z-0.6}{-0.6+z}$
at left and \protect \linebreak{}
$F_{\lambda}|_{R}(z)=\frac{z-\lambda i}{\lambda i+z}$ and $F_{\lambda}|_{R^{c}}(z)=\frac{z-\lambda}{-\lambda+z}$
with $|\lambda-0.6|\ll1$ at right, where $R=\left\{ z:\,|i-z|<\frac{1}{2}\right\} $.
In this case $\Gamma_{F}$ is a Schottky group and $\partial R$ is
contained in a fundamental region of $\Gamma_{F}$. Then $F$ and
$F_{\lambda}$ are quasiconformally conjugated.}
\end{figure}


\section{Appendix: Technical results and constructions\label{sec:Technical-lemmas}}

Some lemmas and propositions about inverse sets and invariance of
pre-discontinuity set and Fatou set of a piecewise conformal map $F$.

\setcounter{lema}{1}

\begin{lema}\label{lem:F-1AB}Let be $A,B\subset\CHat$, then
$F^{-1}(A\cup B)=F^{-1}(A)\cup F^{-1}(B)$.\end{lema} 

\begin{proof}
\[
\begin{array}{rcl}
F^{-1}(A\cup B) & = & \bigcup_{m=1}^{M}\left(\big(F|_{R_{m}}\big)^{-1}\big(A\cup B\big)\cap R_{m}\right)\\
 & = & \bigcup_{m=1}^{M}\left(\big(F|_{R_{m}}\big)^{-1}\big(A\big)\cup\big(F|_{R_{m}}\big)^{-1}\big(B\big)\cap R_{m}\right)\\
 & = & \bigcup_{m=1}^{M}\left(\big(F|_{R_{m}}\big)^{-1}\big(A\big)\cap R_{m}\right)\cup\bigcup_{m=1}^{M}\left(\big(F|_{R_{m}}\big)^{-1}\big(B\big)\cap R_{m}\right)\\
 & = & F^{-1}(A)\cup F^{-1}(B)
\end{array}
\]
\end{proof}

\begin{lema}\label{lem:F-1Spid}$F^{-1}(\mathcal{PD}_{n}(F))=F^{-n+1}(\partial R)\cup\dots\cup F^{-1}(\partial R)=\mathcal{PD}_{n+1}(F)-\partial R$,
for all $n\in\N$.\end{lema}

\begin{proof}Taking $\partial R=\mathcal{PD}_{0}(F),$ then $F^{-1}\big(\mathcal{PD}_{0}(F)\big)=F^{-1}(\partial R)=\mathcal{PD}_{1}(F)-\partial R$.

Now we suppose $F^{-1}\big(\mathcal{PD}_{n-1}(F)\big)=\mathcal{PD}_{n}(F)-\partial R$,
then

\[
\begin{array}{rcl}
F^{-1}\big(\mathcal{PD}_{n}(F)\big) & = & F^{-1}\big(F^{-n}(\partial R)\cup \mathcal{PD}_{n-1}(F)\big)\\
 & = & F^{-1}\big(F^{-n}(\partial R)\big)\cup F^{-1}\big(\mathcal{PD}_{n-1}(F)\big)\\
 & = & F^{-n-1}(\partial R)\cup\big(\mathcal{PD}_{n}(F)-\partial R\big)\\
 & = & \mathcal{PD}_{n+1}(F)-\partial R
\end{array}
\]
\end{proof}

\begin{prop}\label{prop:InvariantSets}The Fatou set is forward invariant
and the pre-discontinuity set is backward invariant.\end{prop}

\begin{proof}Let be $U$ a Fatou component. Since $F|_{U}$ is a
M\"obius transformation, $F(U)$ is open. Also we have that $U=\mathrm{int}(\mathcal{C}_{s})$,
the interior of the set of points with same itinerary $s\in\Sigma_{M}$,
and $F$ is semi-conjugated with the shift $\sigma$, then $F(U)\subset\mathrm{int}(\mathcal{C}_{\sigma(s)})$.
Therefore $F\big(\CHat-\mathcal{PD}(F)\big)\subset\CHat-\mathcal{PD}(F)$.

Let $z\in \mathcal{PD}(F)$ and $w\in F^{-1}(\{z\})$. Suppose that $w\notin \mathcal{PD}(F)$,
by forward invariance of the Fatou set $F(w)=z\notin \mathcal{PD}(F)$, a
contradiction. Therefore $F^{-1}\big(\mathcal{PD}(F)\big)\subset \mathcal{PD}(F)$.\end{proof}

We denote by $\mathcal{H}(\CHat)$ the space of compact subsets
of $\CHat$ with the Hausdorff topology and metric induced by a
spherical metric. $\mathcal{H}(\CHat)$ is a complete metric space
(see for example \cite{Nad}). We recall the convergence characterization
on $\mathcal{H}(\CHat)$.

\begin{defi}A sequence of compacts $K_{n}$ \textbf{converge} to
$K$ on $\mathcal{H}(\CHat)$ if
\begin{enumerate}
\item Every neighborhood $\mathcal{N}_{z}$ of a point $z\in K$ intersects
infinitely many $K_{n}$.
\item If every neighborhood $\mathcal{N}_{z}$ of $z$ intersects infinitely
many $K_{n}$, then $z\in K$.
\end{enumerate}
And we denote this as $K_{n}\rightarrow K$.\end{defi}

In the following, we demonstrate some useful lemmas about convergence
in $\mathcal{H}(\CHat)$. Let be $A_{n}\rightarrow A$ and $B_{n}\rightarrow B$
convergent sequences in $\mathcal{H}(\CHat)$.

\begin{lema}\label{lem:HausConvUnion}$A_{n}\cup B_{n}\rightarrow A\cup B$.\end{lema}

\begin{proof}Let $z\in A\cup B$. Then every neighborhood $\mathcal{N}_{z}$
intersects infinitely many $A_{n}$ or infinitely many $B_{n}$, that
is, intersects infinitely many $A_{n}\cup B_{n}$.

On the other hand, if every neighborhood $\mathcal{N}_{z}$ intersects
infinitely many $A_{n}\cup B_{n}$, then $z\in A$ or $z\in B$.\end{proof}

\begin{lema}\label{lem:HausConvInter}$A_{n}\cap B_{n}\rightarrow A\cap B-Y$
and $Y\subset\partial(A\cap B)$, where\linebreak{}
 $Y=\{z\in A\cap B:\,\,\exists\,\mathcal{N}_{z}\,\mathrm{that\,intersects}$\textbf{$\mathrm{\,finitely\,many\,\,}A_{n}\cap B_{n}\}$}
is the set of isolated points.\end{lema}

\begin{proof}Let $z\in A\cap B-Y$. Then every neighborhood $\mathcal{N}_{z}$
intersects infinitely many $A_{n}$ and infinitely many $B_{n}$ but
$z\notin Y$, that is, intersects infinitely many $A_{n}\cap B_{n}$.

On the other hand, if every neighborhood $\mathcal{N}_{z}$ intersects
infinitely many $A_{n}\cap B_{n}$, then $z\in A$ and $z\in B$,
but $z\notin Y$.

If $z\in Y\subset A\cap B$, then exists $\mathcal{N}_{z}$ such that
intersects finitely many $A_{n}\cap B_{n}$. Then $\mathcal{N}_{z}$
intersects infinitely many $(A_{n}\cap B_{n})^{c}$. Since $(A_{n}\cap B_{n})^{c}=A_{n}^{c}\cup B_{n}^{c}\subset\overline{A_{n}^{c}}\cup\overline{B_{n}^{c}}$,
$\mathcal{N}_{z}$ intersects infinitely many $\overline{A_{n}^{c}}\cup\overline{B_{n}^{c}}$
and, because Lemma 1, $z\in\overline{A^{c}}\cup\overline{B^{c}}=\overline{(A\cap B)^{c}}$.
Finally, $z\in(A\cap B)\cap\overline{(A\cap B)^{c}}=\partial(A\cap B)$.\end{proof}

\begin{lema}\label{lem:HausConvHom}If $f:\CHat\rightarrow\CHat$
is continuous bijective, then $f(A_{n})\rightarrow f(A)$.\end{lema}

\begin{proof}Let $w\in f(A)$. Then exists $z\in A$ such that $f(z)=w$.
We can construct a sequence $z_{n}\rightarrow z$ with $z_{n}\in A_{n}$
because $A_{n}\rightarrow A$. As $f$ is continuous, we have $f(z_{n})\rightarrow f(z)=w$,
then every neighborhood $\mathcal{N}_{w}$ intersects infinitely many
$f(A_{n})$.

On the other hand, if every neighborhood $\mathcal{N}_{w}$ intersects
infinitely many $f(A_{n})$, then we can take a sequence $w_{n}=f(z_{n})\in f(A_{n})$
such that $w_{n}\rightarrow w$. As $f$ is continuous bijective,
$f^{-1}(w_{n})=z_{n}\rightarrow f^{-1}(w)$. Let $z=f^{-1}(w)$, then
every neighborhood $\mathcal{N}_{z}$ intersects infinitely many $A_{n}$
and, by hypothesis, $z\in A$. Finally, $w=f(z)\in f(A)$.\end{proof}

\begin{lema}\label{lem:HausConvJordanCurve}If $C_{n}\rightarrow C$
where each $C_{n}$ and $C$ are compact subsets of $\CHat$
homeomorphic to circles, then $D_{n}\rightarrow D$ and $E_{n}\rightarrow E_{n}$,
where $D_{n}$ and $D$ are the closure of the interior domains of $C_{n}$
and $C$, respectively, and $E_{n}$ and $E$ are the closure of the
exterior domains of $C_{n}$ and $C$, respectively.\end{lema}

\begin{proof}Let $z\in D$. If $z\in C$, then every neighborhood
$\mathcal{N}_{z}$ intersects infinitely many $D_{n}$, because $C_{n}\subset D_{n}$.
If $z\in\mathrm{int}(D)$ and exists one neighborhood $\mathcal{N}'_{z}\subset\mathrm{int}(D)$
such that intersects finitely many $D_{n}$, then intersects infinitely
many $E_{n}$, because $D_{n}^{c}\subset E_{n}$. That is, $z$ is
in the exterior of $C_{k}$ for almost all $n$ but $z$ is in the
interior of $C$. In consequence $C_{n}\nrightarrow C$, leading us
to a contradiction. Therefore, every neighborhood $\mathcal{N}_{z}$
intersects infinitely many $D_{n}$.

Analogously, if $z\in E$ then every neighborhood $\mathcal{N}_{z}$
intersects infinitely many $E_{n}$.

Let be $z$ such that every neighborhood $\mathcal{N}_{z}$ intersects
infinitely many $D_{n}$. If $\mathcal{N}_{z}$ intersects infinitely
many $C_{n}$, then $z\in C\subset D$. If $\mathcal{N}_{z}$ intersects
finitely many $C_{n}$, then $\mathcal{N}_{z}$ must intersects infinitely
many $\mathrm{int}(D_{n})$ and finitely many $E_{n}$. Then $z\notin E$,
that is, $z\in E^{c}\subset D$.

Analogously, if $z$ is such that every neighborhood $\mathcal{N}_{z}$
intersects infinitely many $E_{n}$, then $z\in E$.\end{proof}

About the $\alpha$-limit, we show in next proposition that is contained
in the set of the limit points of backward iterations of $\partial R$
in $\mathcal{H}(\CHat)$, hence its name.

\begin{prop}\label{prop:AlphaLimit1}$\alpha(F)\subset\underset{n\rightarrow\infty}{\lim}\overline{F^{-n}(\partial R)}$
in $\mathcal{H}(\CHat)$.\end{prop}

\begin{proof}Let $z\in\alpha(F)$. By definition of closure, every
neighborhood $\mathcal{N}_{z}$ intersects $\bigcup_{n\geq0}F^{-n}(\partial R)=\bigcup_{n\geq0}\overline{F^{-n}(\partial R)}$,
because $\mathcal{PD}(F)=\overline{\bigcup_{n\geq0}F^{-n}(\partial R)}$.
Suppose that exists a neighborhood $\mathcal{N}'_{z}$ such that intersects
finitely many $\overline{F^{-n}(\partial R)}$. Then exists $N$ such
that for all $n>N$ every neighborhood $\mathcal{N}''_{z}\subset\mathcal{N}'_{z}$
does not intersects $\overline{F^{-n}(\partial R)}$ but $\mathcal{PD}_{N}(F)\cap\mathcal{N}''{}_{z}\neq\emptyset$.
Then, because $\mathcal{PD}_{N}(F)$ is a closed set, $z\in \mathcal{PD}_{N}(F)$,
contradicting the hypothesis. Therefore, every neighborhood $\mathcal{N}_{z}$
must intersect infinitely many $\overline{F^{-n}(\partial R)}$.\end{proof}

\begin{rema}About the structure of the pre-discontinuity sets of a piecewise conformal
maps $F$, we can observe the following combinatorics. Let $f_{m}=F|_{R_{m}}$ and
$\Sigma_{M}(k)=\left\{ 1,\dots,M\right\} ^{k}=\left\{ (m_{1},\dots,m_{k}):\,\,m_{i}\in\left\{ 1,\dots,M\right\} \,\right\} $,
the set of words of $M$ symbols of length $k$.

Defining $C_{1,m}=f_{m}^{-1}(\partial R)\cap R_{m}$ for $m=1,\dots,M$,
we have 
\[
F^{-1}(\partial R)=\bigcup_{m=1}^{M}f_{m}^{-1}(\partial R)\cap R_{m}=C_{1,1}\cup\dots\cup C_{1,M}=\bigcup_{t\in\Sigma_{M}(1)}C_{1,t}
\]

Iteratively, we construct 
\[
\begin{array}{ccl}
F^{-n}(\partial R) & = & F^{-1}\big(F^{-(n-1)}(\partial R)\big)\\
 & = & \bigcup_{m=1}^{M}\big(f_{m}^{-1}(\bigcup_{s\in\Sigma_{M}(n-1)}C_{n-1,s})\big)\cap R_{m}\\
 & = & \bigcup_{s\in\Sigma_{M}(n-1)}\big(C_{n,s1}\cup\dots\cup C_{n,sM}\big)\\
 & = & \bigcup_{t\in\Sigma_{M}(n)}C_{n,t}
\end{array}
\]
where $C_{n,sm}=f_{m}^{-1}(C_{n-1,s})\cap R_{m}$, $s=(m_{1},\dots,m_{n-1})$
and $sm=(m_{1},\dots,m_{n-1},m)$.

Finally, we list some properties about the sets $C_{n,t}$:
\begin{enumerate}
\item For every $n$, each $C_{n,t}$ is a finite union of curve segments
and points, or an empty set.
\item $F(C_{n,sm})\subset C_{n-1,s}$, because $C_{n,sm}=f_{m}^{-1}(C_{n-1,s})\cap R_{m}$
for some $m$.
\item If $C_{n,t}\neq\emptyset$ for some $t\in\Sigma_{M}(n)$, then $F^{n}(C_{n,t})\subset\partial R$.
\end{enumerate}
\end{rema}

\end{document}